\documentclass[a4paper,10pt,reqno]{amsart}
\usepackage{amsmath,amsfonts,amsthm,amssymb,color}
\usepackage[T1]{fontenc}
\usepackage{pdfsync}
\usepackage{csquotes}
\usepackage{graphicx}
\usepackage{pstricks}
\usepackage{lmodern}

\newcommand{\change}[1]{\textcolor{black}{#1}}

\newcommand{\1}{{\bf 1}}

\newcommand{\rti}{\tilde{r}}
\newcommand{\xiti}{\tilde{\xi}}
\newcommand{\etati}{\tilde{\eta}}

\newcommand{\R}{\mathbb R}

\newcommand{\cb}{\mathcal B}
\newcommand{\cac}{\mathcal C}
\newcommand{\cd}{\mathcal D}
\newcommand{\ce}{\mathcal E}
\newcommand{\cf}{\mathcal F}

\newcommand{\ch}{\mathcal H}

\newcommand{\cn}{\mathcal N}

\newcommand{\cs}{\mathcal S}
\newcommand{\cw}{\mathcal W}

\newcommand{\al}{\alpha}

\newcommand{\ga}{\gamma}
\newcommand{\ka}{\kappa}
\newcommand{\la}{\lambda}

\newcommand{\si}{\sigma}

\newcommand{\lln}{\left|}
\newcommand{\rrn}{\right|}

\newtheorem{theorem}{Theorem}[section]

\newtheorem{corollary}[theorem]{Corollary}
\newtheorem{definition}[theorem]{Definition}
\newtheorem{lemma}[theorem]{Lemma}
\newtheorem{proposition}[theorem]{Proposition}
\theoremstyle{remark}
\newtheorem{remark}[theorem]{Remark}

\date{\today}

\begin{document}

\

\begin{center}
{\large\textbf{
A non-linear wave equation with fractional perturbation
}}\\~\\
Aur\'elien Deya\footnote{Institut \'Elie Cartan, Universit\' e de Lorraine, BP 70239, 54506 Vandoeuvre-l\`es-Nancy, France. Email: {\tt aurelien.deya@univ-lorraine.fr}}
\end{center}

\bigskip

{\small \noindent {\bf Abstract:} We study a $d$-dimensional wave equation model ($d\in \{2,3\}$) with quadratic non-linearity and stochastic forcing given by a space-time fractional noise. Two different regimes are exhibited, depending on the Hurst parameter $H=(H_0,\ldots,H_d) \in (0,1)^{d+1}$ of the noise: if $\sum_{i=0}^d H_i >  d-\frac12$, then the equation can be treated directly, while in the case $d-\frac34<\sum_{i=0}^d H_i\leq  d-\frac12$, the model must be interpreted in the Wick sense, through a renormalization procedure. \\
\indent Our arguments essentially rely on a fractional extension of the considerations of \cite{gubinelli-koch-oh} for the two-dimensional white-noise situation, and more generally follow a series of investigations related to stochastic wave models with polynomial perturbation. 

\bigskip



\section{Introduction and main results}

In this paper, we propose to study the following non-linear stochastic wave equation: 
\begin{equation}\label{1-d-quadratic-wave}
\left\{
\begin{array}{l}
\partial^2_t u-\Delta u + \rho^2  u^2=\dot{B} \, , \quad \quad t\in [0,T] \, , \ x\in \R^d \, ,\\
u(0,.)=\phi_0 \ , \ \partial_t u(0,.)=\phi_1 \, ,
\end{array}
\right.
\end{equation}
where $\phi_0,\phi_1$ are (deterministic) initial conditions in an appropriate Sobolev space, $\rho:\R^d \to \R$ is a smooth (deterministic) function with support included in a bounded domain $D\subset \R^d$, $d\in \{2,3\}$, and $\dot{B}\triangleq\partial_t\partial_{x_1} \cdots \partial_{x_{d}} B$ for some space-time fractional Brownian motion $B=B^{H}$ of Hurst index $H=(H_0,H_1,\ldots,H_{d}) \in (0,1)^{d+1}$. For the sake of clarity, let us here recall the specific definition of this process:

\begin{definition}
Fix a dimension parameter $d\geq 1$, as well as a complete filtered probability space $(\Omega,\mathcal{F},\mathbb{P})$. For any $H=(H_0,H_1,\ldots,H_d) \in (0,1)^{d+1}$, a centered Gaussian process $B:\Omega \times ([0,T]\times \R^d) \to \R$ is called a space-time fractional Brownian motion (or a fractional Brownian sheet) of Hurst index $H$ if its covariance function is given by the formula
$$\mathbb{E} \big[ B(s,x_1,\ldots,x_d)B(t,y_1,\ldots,y_d)\big] = R_{H_0}(s,t) \prod_{i=1}^d R_{H_i}(x_i,y_i) \ ,$$
where
$$R_{H_i}(x,y)\triangleq\frac12 (|x|^{2H_i}+|y|^{2H_i}-|x-y|^{2H_i}) \ .$$
In particular, a space-time fractional Brownian motion of Hurst index 
$$H=(\frac12,\ldots,\frac12)$$
is a Wiener process (and in this case the derivative $\dot{B}$ is a space-time white noise).
\end{definition}

\

Since the pioneering works of Mandelbrot and Van Ness, fractional noises have been considered as very natural stochastic perturbation models, that offer more flexibility than classical white-noise-driven equations. The involvement of fractional inputs first occured in the setting of standard differential equations and, even in this simple context, the procedure is known to raise numerous difficulties due to the non-martingale nature of the process. Sophisticated alternatives to It{\^o} theory must then come into the picture, whether fractional calculus, Malliavin calculus or rough paths theory, to mention just the most standard methods.  

\

More recently, fractional (multiparameter) noises have also appeared within SPDE models. A first widely-used example is given by white-in-time colored-in-space Gaussian noises, that can be treated in the classical framework of Walsh's martingale-measure theory \cite{walsh}, or with Da Prato-Zabczyk's infinite-dimensional approach to stochastic calculus \cite{daprato-zabczyk}. Such noise models have thus been applied to a large class of PDE dynamics, and the properties of the solutions to the resulting SPDEs are often well understood (see \cite{daprato-zabczyk} and the numerous references therein).

\

SPDEs involving a fractional-in-time noise are much more delicate to handle (Walsh and Da Prato-Zabczyk theories no longer apply in this case), and the related literature is in fact very scarce:

\smallskip

\noindent
$\bullet$ In the parabolic setting, one can first mention \cite{tindel-tudor-viens} for the study of a homogeneous equation with additive fractional Brownian motion, and the series of papers \cite{hu,hu-lu-nualart,hu-nualart-song} for the analysis of a linear multiplicative perturbation of the heat equation. Pathwise approaches to the parabolic fractional problem have also been considered in \cite{RHE,REE} using rough-paths ideas, and in \cite{deya,deya2} with the formalism of Hairer's theory of regularity structures.

\smallskip

\noindent
$\bullet$ For the wave equation, and to the best of our knowledge, the results are so far limited to the analysis of the specific one-dimensional ($d=1$) situation \cite{balan-jolis-quer,caithamer,erraoui-ouknine-nualart, quer-tindel}, and to the study of affine models when $d\geq 2$: the homogeneous equation with additive fractional noise in \cite{balan-tudor} and multiplicative linear noise in \cite{balan} (when the time-fractional order satisfies $H_0>1/2$ and the space covariance structure is given by a Riesz kernel of order $\al >d-2$).  

\

In brief, SPDEs, and especially stochastic hyperbolic equations, driven by a space-time fractional noise remain a widely-open field at this point. Note in particular that the wave-equation case cannot be treated within the recently-introduced framework of regularity structures (\cite{hai-14}), due to the lack of regularization properties for the wave kernel with respect to space-time Sobolev topologies.

\

With this general background in mind, let us now go back to the consideration of equation (\ref{1-d-quadratic-wave}). Our approach to the model will directly follow a series of investigations \cite{bourgain,burq-tzvetkov,gubinelli-koch-oh,oh-thomann,thomann} devoted to the study of stochastic wave (or Schr{\"o}dinger) equations involving a polynomial drift term. Our study can more specifically be seen as a fractional extension of the results of \cite{gubinelli-koch-oh} for the white-noise situation. In the last five references, and in our study as well, the strategy to handle the equation relies on a central ingredient that is often referred to as the Da Prato-Debussche's trick. Roughly speaking, it consists in regarding the solution $u$ of (\ref{1-d-quadratic-wave}) as some \enquote{perturbation} of the solution $\Psi$ to the associated \enquote{free} equation
\begin{equation}\label{equation-psi}
\left\{
\begin{array}{l}
\partial^2_t \Psi -\Delta \Psi=\dot{B} \, , \quad \quad  t\in [0,T] \, , \ x\in \R^d \, ,\\
\Psi(0,.)=0\ , \ \partial_t \Psi(0,.)=0  \, .
\end{array}
\right. 
\end{equation}
In fact, staying at a heuristic level, observe that the difference process $v\triangleq u-\Psi$ satisfies (morally) the equation
\begin{equation}\label{equa-deter-base-intro}
\left\{
\begin{array}{l}
\partial^2_t v -\Delta v+ \rho^2 (v^2+2v\cdot \Psi+\Psi^2)=0\, , \quad \quad t\in [0,T] \, , \ x\in \R^d \, ,\\
v(0,.)=\phi_0 \ , \ \partial_t v(0,.)=\phi_1 \, .
\end{array}
\right.
\end{equation}
The key of the method then lies in the fact that, once endowed with a good understanding of the pair $(\Psi,\Psi^2)$, equation (\ref{equa-deter-base-intro}) turns out to be much more tractable than the original equation (\ref{equation-psi}), and can be solved with pathwise arguments. The procedure thus emphasizes the following idea: to some extent, the difficulties behind the analysis of equation (\ref{equation-psi}) reduce to the difficulties in the study of the two processes $\Psi$ and $\Psi^2$. Note in particular that this general approach offers a clear splitting between the stochastic part of the analysis (i.e., the study of $(\Psi,\Psi^2)$), and the deterministic part of the problem (i.e., the pathwise study of (\ref{equa-deter-base-intro})). This decomposition is very reminiscent of the spirit of rough paths (or regularity structures) theory, where the solution of the problem is also built in a deterministic way around a stochastically-constructed object.  

\

The solution $\Psi$ of (\ref{equation-psi}) is therefore expected to play a fundamental role in the analysis, and a first step consists of course in providing a clear definition of this process (we recall that the space-time fractional setting is not exactly standard). To this end, we will appeal to a natural approximation procedure and construct $\Psi$ as the limit of a sequence of (classical) solutions driven by a smooth approximation $\dot{B}_n$ of $\dot{B}$ (or equivalently a smooth approximation $B_n$ of $B$). Just as in \cite{deya,deya2}, the approximation that we will consider here is derived from the so-called harmonizable representation of the space-time fractional Brownian motion (see e.g. \cite{samo-taqqu}), that is the formula (valid for every $H=(H_0,\ldots,H_d)\in (0,1)^{d+1}$)  
\begin{equation*}
B(t,x_1,\ldots,x_d)= c_{H} \int_{\xi\in \R}\int_{\eta \in \R^d} \widehat{W}(d\xi,d\eta) \frac{e^{\imath t\xi}-1}{|\xi|^{H_0+\frac12}} \prod_{i=1}^d \frac{e^{\imath x_i \eta_i }-1}{|\eta_i|^{H_i+\frac12}} \ ,
\end{equation*}
where $c_H >0$ is a suitable constant and $\widehat{W}$ stands for the Fourier transform of a space-time white noise in $\R^{d+1}$, defined on some complete filtered probability space $(\Omega,\mathcal{F},\mathbb{P})$. The approximation $(B_n)_{n\geq 1}$ of $B$ is then defined as
\begin{equation}
B_n(t,x_1,\ldots,x_d)\triangleq c_{H} \int_{|\xi|\leq 2^n}\int_{|\eta|\leq 2^n} \widehat{W}(d\xi,d\eta) \frac{e^{\imath t\xi}-1}{|\xi|^{H_0+\frac12}} \prod_{i=1}^d \frac{e^{\imath x_i \eta_i }-1}{|\eta_i|^{H_i+\frac12}} \ .
\end{equation}
It is readily checked that for all fixed $H=(H_0,H_1,\ldots,H_d)\in (0,1)^{d+1}$ and $n\geq 1$, the so-defined process $B_n$ indeed corresponds to a smooth function (almost surely). Accordingly, the associated equation
\begin{equation}\label{equation-psi-n}
\left\{
\begin{array}{l}
\partial^2_t \Psi_n -\Delta \Psi_n=\dot{B}_n \, , \quad \quad  t\in [0,T] \, ,  \ x\in \R^d \, ,\\
\Psi_n(0,.)=0\ ,  \ \partial_t \Psi_n(0,.)=0 \, , 
\end{array}
\right.
\end{equation}
falls within the class of standard hyperbolic systems, for which a unique (global) solution $\Psi_n$ is known to exist. Our first result now reads as follows:

\begin{proposition}\label{prop:regu-psi}
\change{Let $d\geq 1$ and $\rho:\R^d \to \R$ be a smooth compactly-supported function. Then,} for every $(H_0,H_1,\ldots,H_d)\in (0,1)^{d+1}$, \change{$(\rho\Psi_n)_{n\geq 1}$} is a Cauchy sequence in the space $L^p(\Omega;L^\infty([0,T];\mathcal{W}^{-\al,p}(\change{\R^d})))$, for all $p\geq 2$ and 
\begin{equation}\label{regu}
\al >d-\frac12- \sum_{i=0}^d H_i \ .
\end{equation}
In particular,  $(\change{\rho \Psi_n})_{n\geq 1}$ converges to a limit in $L^p(\Omega;L^\infty([0,T];\mathcal{W}^{-\al,p}(\change{\R^d})))$, that we denote by $\change{\rho\Psi}$. 
\end{proposition}

This approach of a fractional equation via a regularization procedure is of course a standard strategy, that is also used for instance in rough paths or regularity structure theory (observe that the interpretation of the equation in \cite{hu-lu-nualart} leans on an approximation method as well). 

\smallskip

\begin{remark}
In \cite{balan-tudor}, the authors tackle the fractional model (\ref{equation-psi}) using a Malliavin-calculus approach, which provides an interpretation and a solution of the equation that may be considered as more intrinsic (since it does not depend on any approximation of the noise). In fact, we think that this Malliavin-calculus solution to (\ref{equation-psi}) could be identified with the limit process $\Psi$ exhibited in Proposition \ref{prop:regu-psi}, but we will not dwell on this identification procedure, since we find it relatively removed from the purpose of our analysis and also because it would require the introduction of the whole Malliavin-calculus framework. Observe however that the results of \cite{balan-tudor} also highlight the threshold $\sum_{i=0}^d H_i= d-\frac12$  (with the additional assumption $H_0>\frac12$) for $\Psi$ to be either a function or a distribution. 
\end{remark}

\smallskip

Based on Proposition \ref{prop:regu-psi}, the limit process $\Psi$ will therefore be considered (almost surely) as a function when $\sum_{i=0}^d H_i> d-\frac12$ and as a distribution otherwise. In the latter situation, and when turning to the study of the auxiliary equation (\ref{equa-deter-base-intro}), one must then face the problem of interpreting the product $\Psi^2$. Just as in \cite{gubinelli-koch-oh,oh-thomann}, we will actually understand this product in the Wick sense, which, again, can be made rigorous through an approximation method, combined with a renormalization procedure:

\begin{proposition}\label{prop:regu-psi-order-two}
\change{Let $d\geq 1$ and $\rho:\R^d \to \R$ be a smooth compactly-supported function. Also,} let $(H_0,H_1,\ldots,H_d)\in (0,1)^{d+1}$ such that 
\begin{equation}\label{constraint-h-i}
d-\frac34 < \sum_{i=0}^d H_i \leq  d-\frac12 \ ,
\end{equation}
and consider the Wick-renormalized product $\widehat{\mathbf{\Psi}}^{\mathbf{2}}_n(t,y)\triangleq \Psi_n(t,y)^2-\si_n(t,y)$, with $\si_n(t,y) \triangleq \mathbb{E}\big[ \Psi_n(t,y)^2\big]$.
Then $(\change{\rho^2 \widehat{\mathbf{\Psi}}^{\mathbf{2}}_n})_{n\geq 1}$ is a Cauchy sequence in the space $L^p(\Omega;L^\infty([0,T];\mathcal{W}^{-2\al,p}(\change{\R^d})))$,
for all $p\geq 2$ and 
\begin{equation}\label{regu-2}
\al >d-\frac12- \sum_{i=0}^d H_i \ .
\end{equation}
In particular,  $(\change{\rho^2 \widehat{\mathbf{\Psi}}^{\mathbf{2}}_n})_{n\geq 1}$ converges to a limit in $L^{p}(\Omega;L^\infty([0,T];\mathcal{W}^{-2\al,p}(\change{\R^d})))$, that we denote by $\change{\rho^2 \widehat{\mathbf{\Psi}}^{\mathbf{2}}}$.
\end{proposition}

\smallskip

Two distinct treatments of the problem (corresponding to the two regimes $d-\frac12-\sum_{i=0}^d H_i < \al <0$ and $d-\frac12 -\sum_{i=0}^d H_i \geq \al \geq 0$ in Proposition \ref{prop:regu-psi}) are thus to occur in our analysis, with a clear transition phenomenon regarding the interpretation of the product $\Psi^2$ and the need for renormalization. In order to encompass these two regimes into a single framework, let us slightly extend the formulation of (\ref{equa-deter-base-intro}) and consider the more general (deterministic) equation
\change{\begin{equation}\label{equa-deter-base}
\left\{
\begin{array}{l}
\partial^2_t v -\Delta v+ \rho^2 v^2+(\rho v)\cdot \mathbf{\Pi}^{\mathbf{1}}+\mathbf{\Pi}^{\mathbf{2}}=0\, , \quad \quad  t\in [0,T] \, , \ x\in \R^d \, ,\\
v(0,.)=\phi_0 \ , \ \partial_t v(0,.)=\phi_1 \, ,
\end{array}
\right.
\end{equation}}
where the two \enquote{parameters} $\mathbf{\Pi}^{\mathbf{1}}$ and $\mathbf{\Pi}^{\mathbf{2}}$ will be either functions or distributions in suitable Sobolev spaces. Our interpretation of the model (\ref{1-d-quadratic-wave}) can now be expressed as follows:

\begin{definition}
Let $\Psi$ and $\widehat{\mathbf{\Psi}}^{\mathbf{2}}$ be the processes defined in Proposition \ref{prop:regu-psi} and Proposition \ref{prop:regu-psi-order-two}.

\smallskip

\noindent
$(i)$ A stochastic process $(u(t,x))_{t\in [0,T],x\in \R^d}$ is said to be a solution (on $[0,T]$) of the equation
\begin{equation}\label{1-d-quadratic-wave-not-renormalized-1}
\left\{
\begin{array}{l}
\partial^2_t u-\Delta u + \rho^2  u^2=\dot{B} \, , \quad \quad  t\in [0,T] \, , \ x\in \R^d \, ,\\
u(0,.)=\phi_0 \ , \ \partial_t u(0,.)=\phi_1 \, ,
\end{array}
\right.
\end{equation}
if, almost surely, $\Psi$ is a function and the auxiliary process $v:=u-\Psi$ is a mild solution (on $[0,T]$) of Equation (\ref{equa-deter-base}) with $\mathbf{\Pi}^{\mathbf{1}}\triangleq 2\change{\rho}\Psi$ and $\mathbf{\Pi}^{\mathbf{2}}\triangleq \change{\rho^2}\Psi^2$.

\smallskip

\noindent
$(ii)$ A stochastic process $(u(t,x))_{t\in [0,T],x\in \R^d}$ is said to be a solution (on $[0,T]$) of the Wick-renormalized equation
\begin{equation}\label{1-d-quadratic-wave-renormalized}
\left\{
\begin{array}{l}
\partial^2_t u-\Delta u + \rho^2  :\! u^2\! \!: \ =\dot{B} \, , \quad \quad  t\in [0,T] \, , \ x\in \R^d \, ,\\
u(0,.)=\phi_0 \ , \ \partial_t u(0,.)=\phi_1 \, ,
\end{array}
\right.
\end{equation}
if, almost surely, the auxiliary process $v:=u-\Psi$ is a mild solution (on $[0,T]$) of Equation (\ref{equa-deter-base}) with $\mathbf{\Pi}^{\mathbf{1}}\triangleq 2\change{\rho}\Psi$ and $\mathbf{\Pi}^{\mathbf{2}}\triangleq \change{\rho^2}\widehat{\mathbf{\Psi}}^{\mathbf{2}}$.
\end{definition}

The results of Section \ref{sec:auxiliary-equation} will in fact allow us to give a clear sense to the notion of a mild solution to (\ref{equa-deter-base}) (with values in a specific space), thus completing the above definition. With this setting in mind, we can finally state the main results of our study.

\begin{theorem}\label{main-theo}
Let $d\in \{2,3\}$ and $(\phi_0,\phi_1) \in \ch^1(\R^d) \times L^2(\R^d)$. Then the following picture holds true:

\smallskip

\noindent
$(i)$
If $\sum_{i=0}^d H_i>d-\frac12$, then, almost surely, there exists a time $T_0>0$ such that the equation (\ref{1-d-quadratic-wave-not-renormalized-1}) admits a unique solution $u$ in the set 
\begin{equation}\label{defi-set-s-t}
\cs_{T_0}\triangleq \, \Psi   + X(T_0) \, , \quad \text{where}\ \ X(T_0)\triangleq L^\infty([0,T_0];\ch^{1}(\R^d)) \, .
\end{equation}

\smallskip

\noindent
$(ii)$ If $d-\frac34<\sum_{i=0}^d H_i\leq  d-\frac12$, then, almost surely, there exists a time $T_0>0$ such that the Wick-renormalized equation (\ref{1-d-quadratic-wave-renormalized})
admits a unique solution $u$ in the set 
\begin{equation}\label{defi-set-s-t-rough}
\cs^s_{T_0}\triangleq \, \Psi   + X^{\frac12}(T_0) \, , \quad \text{where}\ \ X^{\frac12}(T_0)\triangleq L^\infty([0,T_0];\ch^{\frac12}(\R^d)) \, ,
\end{equation}
\end{theorem}

\

Using the continuity properties of the solution $v$ of (\ref{equa-deter-base}) with respect to $(\mathbf{\Pi}^{\mathbf{1}},\mathbf{\Pi}^{\mathbf{2}})$, we will also be able to \enquote{lift} the convergence statements for $\Psi$ and $\Psi^2$ (i.e., the results of Propositions \ref{prop:regu-psi} and \ref{prop:regu-psi-order-two}) at the level of the equation, which will offer the following alternative interpretation of the model:

\begin{theorem}\label{main-theo-lim}
Let $d\in \{2,3\}$ and $(\phi_0,\phi_1) \in \ch^1(\R^d) \times L^2(\R^d)$. Then the following picture holds true:

\smallskip

\noindent
$(i)$
If $\sum_{i=0}^d H_i>d-\frac12$, consider the sequence $(u_n)_{n\geq 1}$ of (classical) solutions to the equation
\begin{equation}\label{1-d-quadratic-wave-not-renormalized}
\left\{
\begin{array}{l}
\partial^2_t u_n-\Delta u_n + \rho^2  u_n^2 =\dot{B}_n \, , \quad \quad  t\in [0,T_0] \, , \ x\in \R^d \, ,\\
u(0,.)=\phi_0 \ , \ \partial_t u(0,.)=\phi_1 \, .
\end{array}
\right.
\end{equation}
Then, almost surely, there exists a time $T_0>0$ and a subsequence of $(u_n)$ that converges in the space $L^\infty([0,T_0];L^2(D))$ to the solution $u$ exhibited in Theorem \ref{main-theo} (item $(i)$).

\smallskip

\noindent
$(ii)$ If $d-\frac34<\sum_{i=0}^d H_i\leq  d-\frac12$, set $\si_n(t)\triangleq \mathbb{E}[\Psi_n(t,x)^2]$ and consider the sequence $(u_n)_{n\geq 1}$ of (classical) solutions to the renormalized equation
\begin{equation}\label{1-d-quadratic-wave-renormalized-lim}
\left\{
\begin{array}{l}
\partial^2_t u_n(t,x)-\Delta u_n (t,x)+ \rho^2(x) ( u_n(t,x)^2-\si_n(t))=\dot{B}_n(t,x) \, , \\
u_n(0,.)=\phi_0 \ , \ \partial_t u_n(0,.)=\phi_1 \, .
\end{array}
\right.
\end{equation}
for $t\in [0,T_0],x\in \R^d$.
Then
\begin{equation}\label{estim-cstt}
\si_n(t)\stackrel{n\to \infty}{\sim} 
\left\lbrace
\begin{array}{ll}
c^1_{H}\, t \, 2^{2n(d-\frac12-\sum_{i=0}^d H_i)}& \quad \text{if} \quad \sum_{i=0}^d H_i<  d-\frac12  \ ,\\
c^2_{H}\, t\,  n & \quad \text{if} \quad \sum_{i=0}^d H_i=  d-\frac12 \ ,
\end{array}
\right.
\end{equation}
for some constants $c^1_{H},c^2_{H}$, and, almost surely, there exists a time $T_0>0$ and a subsequence of $(u_n)$ that converges in the space $L^\infty([0,T_0];\ch^{-\al}(D))$ to the solution $u$ exhibited in Theorem \ref{main-theo} (item $(ii)$), for every $\al >d-\frac12- \sum_{i=0}^d H_i$.
\end{theorem}

As far as we know, Theorems \ref{main-theo} and \ref{main-theo-lim} are the first wellposedness results for a non-linear wave model involving a space-time fractional noise (at least beyond the very specific one-dimensional situation). Observe that the above change-of-regime phenomenon is especially relevant in the fractional setting, where the roughness parameter $H$ can be \enquote{continuously} modified in $(0,1)^{d+1}$ (contrary to the space parameter $d\in \{2,3,\ldots\}$).  

\smallskip

The rest of the paper is devoted to the proof of these successive statements. Let us just conclude this introduction with a few additional remarks about Theorems \ref{main-theo} and \ref{main-theo-lim}. 

\smallskip

\begin{remark}
The consideration of the linear combination $\sum_{i=0}^d H_i$ in the above splitting must be compared with the role of the linear combination $2H_0+\sum_{i=1}^d H_i$ in the study of the fractional heat equation (see e.g. \cite[Theorem 1.2] {deya}). These combinations naturally echo the hyperbolic and parabolic settings, with scaling coefficient $\mathfrak{s}=(1,1,\ldots,1)$ and $\mathfrak{s}=(2,1,\ldots,1)$, respectively.
\end{remark}

\smallskip

\begin{remark}
Of course, the two situations $(i)$ and $(ii)$ in Theorem \ref{main-theo} and Theorem \ref{main-theo-lim} do not cover the whole range of possibilities for the Hurst index $H=(H_0,\ldots,H_d)\in (0,1)^{d+1}$ of the noise. 

The restricting condition $\sum_{i=0}^d H_i >d-\frac34$ is first inherited from our computations towards Proposition \ref{prop:regu-psi-order-two} (as reported in (\ref{constraint-h-i})), where we lean on the possibility to pick $\al <\frac14$ (see (\ref{importance-alpha-1-4})), and, due to (\ref{regu-2}), this can indeed be done only if $\sum_{i=0}^d H_i >d-\frac34$. We suspect that, at the price of a sophisticated refinement of the estimations of Section \ref{subsec:proof-conv-order-two} (starting from a refinement of the transition from (\ref{i-n-s-t-0}) to (\ref{i-n-s-t})), the renormalization result of Proposition \ref{prop:regu-psi-order-two} should in fact remain true up to the critical value $\sum_{i=0}^d H_i =d-1$. This conjecture is essentially based on the results of \cite{deya-wave-2} for the particular dimension $d=2$, where the Wick-renormalization of $\Psi^2$ is shown to be possible up to the critical situation $H_0+H_1+H_2=1$ (see \cite[Propositions 1.3 and 1.4]{deya-wave-2}).

On the other hand, and in light of the assumptions in the subsequent Proposition \ref{prop:deterministic-result-d}, it is clear that the deterministic part of our analysis can only be applied if $\al < \frac13$. Keeping condition (\ref{regu-2}) in mind, this would here lead to the restriction $\sum_{i=0}^d H_i >d-\frac56$, and so, in brief, we think that the \enquote{second-order} results of Theorems \ref{main-theo} and \ref{main-theo-lim} should remain valid if $d-\frac56 <\sum_{i=0}^d H_i \leq d-\frac34$. To our opinion, extending such properties to the case $d-1 <\sum_{i=0}^d H_i \leq d-\frac56$ can only be done through the consideration of higher-order expansions of the equation, as performed in \cite{deya-wave-2} for the particular dimension $d=2$.
\end{remark}

\smallskip

\begin{remark}
The forthcoming proofs (and accordingly the above results) could certainly be extended to more general covariance structures, such as the ones considered for instance in \cite{balan-tudor}. Our arguments are indeed based on a Fourier-type analysis, which suggests that a suitable control on the Fourier transform of the covariance function might be sufficient for the computations to remain valid. Besides, we think that, just as in rough paths or regularity structures results, the above properties are in fact relatively independent of the choice of the approximation $B_n$. For instance, using an appropriate Fourier transformation, the results should be the same when starting from an approximation of the form $B_n\triangleq \varphi_n \ast B$, for a given mollifying sequence $(\varphi_n)_{n\geq 1}$ (the only possible difference may be the value of the constants $c_H^1,c_H^2$ in (\ref{estim-cstt}), as classically observed in regularity structures theory).
\end{remark}

\smallskip

\begin{remark}
For $d=2$, our results cover the white-noise situation $H_0=H_1=H_2=\frac12$, and so we can consider Theorem \ref{main-theo} as a fractional extension of \cite[Theorem 1.1]{gubinelli-koch-oh} in the quadratic case (as far as the non-linearity). Our study thus offers an additional illustration of the flexibility of the general two-step procedure described above (i.e., we first study the free equation (\ref{equation-psi}) and then the auxiliary equation (\ref{equa-deter-base})). Observe that the white-noise situation for $d=2$ corresponds here to a \enquote{border case}, that is a case for which $\sum_{i=0}^d H_i=  d-\frac12$, with specific rate of divergence in (\ref{estim-cstt}). 
\end{remark}

\smallskip

\begin{remark}
As the reader may have guessed it, the involvement of the smooth function $\rho$ in (\ref{1-d-quadratic-wave}) is only meant to bring the computations back to a compact space-time domain (which will be often esssential in the sequel). Thus, our results should morally be read as local results, both in time and in space, for the real \enquote{target} equation, that is the equation with $\rho\equiv 1$. What refrained us to formulate the problem on a torus (just as in \cite{gubinelli-koch-oh,oh-thomann}) is the consideration of the fractional noise, which is more convenient to define and handle on the whole Euclidean space. 
\end{remark}

\

As we already pointed it out, our analysis will be clearly divided into a stochastic and a deterministic part. The organization of the paper will follow this splitting. Section \ref{sec:stochastic} is first devoted to the stochastic analysis, that is the study of $\Psi_n$ and the proof of Propositions \ref{prop:regu-psi} and \ref{prop:regu-psi-order-two}. The estimation (\ref{estim-cstt}) of the renormalization constant (which is directly related to $\Psi_n$) will also be carried out in this section. In Section \ref{sec:auxiliary-equation}, we will focus on the deterministic study of the auxiliary equation (\ref{equa-deter-base}), first in the \enquote{regular} case where $\mathbf{\Pi}^{\mathbf{1}}$ and $\mathbf{\Pi}^{\mathbf{1}}$ are functions (Proposition \ref{prop:deterministic-result-d-i}), and then in the distributional situation (Proposition \ref{prop:deterministic-result-d}). We will finally combine these successive results in Section \ref{sec:proofs-main-theos} in order to derive the proof of Theorem \ref{main-theo} and Theorem \ref{main-theo-lim}.

\

Throughout the paper, and for any normed space $E$, the notation $\cn[v;E]$ will refer to the norm of $v\in E$.

\

\textbf{Acknowledgements.} I am deeply grateful to two anonymous reviewers for their careful reading and their enthusiastic comments about this study.

\

\section{Study of the (stochastic) linear equation}\label{sec:stochastic}

We here propose to tackle the issues related to the solution $\Psi_n$ of the regularized equation \eqref{equation-psi-n}.

\smallskip

For a fixed dimension $d\geq 1$, let us denote by $G$ the Green function associated with the standard $d$-dimensional wave equation and recall that the (space) Fourier transform of $G$ is explicitly given for all $t\geq 0$ and $x\in \R^d$ by the formula 
$$\int_{\R^d} dy \,  e^{-\imath \langle x,y\rangle }G_{t}(y)=\frac{\sin(t |x|)}{|x|} \ .$$
Now, the solution $\Psi_n$ of (\ref{equation-psi-n}) can be written as  
\begin{eqnarray}
\Psi_n(t,x)&=&\int_0^t ds \, (G_{t-s} \ast \dot{B}_n(s))(x)\nonumber \\
&=& c \int_{|\xi|\leq 2^n}\int_{|\eta|\leq 2^n} \widehat{W}(d\xi,d\eta) \frac{\xi}{|\xi|^{H_0+\frac12}} \prod_{i=1}^d \frac{\eta_i}{|\eta_i|^{H_i+\frac12}}\int_0^t ds \int_{\R^d} dy \, G_{t-s}(x-y) e^{\imath \xi s} e^{\imath \langle \eta, y\rangle}\nonumber\\
&=& c \int_{|\xi|\leq 2^n}\int_{|\eta|\leq 2^n} \widehat{W}(d\xi,d\eta) \frac{\xi}{|\xi|^{H_0+\frac12}} \prod_{i=1}^d \frac{\eta_i e^{\imath \eta_i x_i}}{|\eta_i|^{H_i+\frac12}}\, \ga_t(\xi,|\eta|) \ , \label{expr-psi-n}
\end{eqnarray}
where for all $t\geq 0$, $\xi \in \R$ and $r>0$, we define the quantity $\ga_t(\xi,r)$ as
\begin{equation}\label{defi-gamma-t}
\ga_t(\xi,r)\triangleq e^{\imath \xi t}\int_0^t ds \, e^{-\imath \xi s} \frac{\sin(sr)}{r} \ .
\end{equation}
Let us also set $\ga_{s,t}(\xi,r)\triangleq \ga_t(\xi,r)-\ga_s(\xi,r)$.

\smallskip

With these notations in hand, our computations towards Proposition \ref{prop:regu-psi} and Proposition \ref{prop:regu-psi-order-two} will extensively rely on the two following elementary estimates.

\begin{lemma}
For all $0\leq s \leq t $, $\xi\in \R$, $r>0$ and $\ka,\la\in [0,1]$, it holds that 
\begin{equation}\label{estimate-gamma}
|\ga_{s,t}(\xi,r)| 
\lesssim  \min\Big( |\xi|^\ka |t-s|^\ka t^2+|t-s| t  ,\frac{|t-s|}{|\xi|}+\frac{|t-s|^\ka t}{|\xi|^{1-\ka}},\frac{ |t-s|^\ka \{r^\ka+|\xi|^\ka\} t^{\la(1-\ka)} }{r\, ||\xi|-r|^{1-\la (1-\ka)}} \Big) \ .
\end{equation}
\end{lemma}

\begin{proof}
First, one has obviously
\begin{eqnarray*}
|\ga_{s,t}(\xi,r)|&\lesssim & |e^{\imath \xi t}-e^{\imath \xi s}|\bigg| \int_0^t du \, e^{-\imath \xi u} \frac{\sin(u r)}{r}\bigg|+\bigg|\int_s^t du \, e^{-\imath \xi u} \frac{\sin(u r)}{r}\bigg|\\
& \lesssim & |\xi|^\ka |t-s|^\ka t^2 +|t-s|t  \ .
\end{eqnarray*}
Then observe that
$$\ga_t(\xi,r)=e^{\imath \xi t}\int_0^t ds \, e^{-\imath \xi s} \frac{\sin(s r)}{r}=-\frac{\sin(tr)}{\imath \xi r}+\frac{e^{\imath \xi t}}{\imath \xi} \int_0^t ds \, e^{-\imath \xi s} \cos(sr) \ ,$$
which readily entails $|\ga_{s,t}(\xi,r)|\lesssim \frac{|t-s|}{|\xi|}+\frac{|t-s|^\ka t}{|\xi|^{1-\ka}}$. Finally, it can be checked that
\begin{equation}\label{exact-expression-gamma}
\ga_t(\xi,r)=\frac{1}{2r} \bigg[ \frac{e^{\imath rt}-e^{\imath \xi t}}{\xi-r} -\frac{e^{-\imath rt}-e^{\imath \xi t}}{\xi+r}  \bigg] \ ,
\end{equation}
which easily leads to
$$|\ga_{s,t}(\xi,r)| \lesssim \frac{ |t-s|^\ka \{r^\ka+|\xi|^\ka\} t^{\la(1-\ka)} }{r\, ||\xi|-r|^{1-\la (1-\ka)}} \ .$$
\end{proof}

\begin{corollary}\label{coro:tech}
For all $0\leq s \leq t\leq 1$, $H\in (0,1)$, $r>0$, $\varepsilon \in (0,1)$ and $\ka \in [0,\min(H, \frac{1-\varepsilon}{2}))$, it holds that
$$\int_{\R} d\xi\,  \frac{|\ga_{s,t}(\xi,r)|^2}{|\xi|^{2H-1}} \lesssim |t-s|^{2\ka}\min\Big(1,\frac{1}{r^{2+2(H-\ka)}}+\frac{1}{r^{1+2(H-\ka)-\varepsilon}} \Big) \ .$$ 
\end{corollary}

\begin{proof}
The two bounds follow from (\ref{estimate-gamma}). First,
\begin{align*}
&\int_{\R} d\xi\,  \frac{|\ga_{s,t}(\xi,r)|^2}{|\xi|^{2H-1}} \lesssim |t-s|^{2\ka}\bigg[ \int_{|\xi|\leq 1}  \frac{d\xi}{|\xi|^{2H-1}} +\int_{|\xi|\geq 1} \frac{d\xi}{|\xi|^{2(H-\ka)+1}}\bigg] \lesssim |t-s|^{2\ka} \ .
\end{align*}
Then consider the decomposition
$$\int_{\R} d\xi\,  \frac{|\ga_{s,t}(\xi,r)|^2}{|\xi|^{2H-1}}=\int_{||\xi|-r| \geq \frac{|\xi|}{2} } d\xi\,  \frac{|\ga_{s,t}(\xi,r)|^2}{|\xi|^{2H-1}}+\int_{||\xi|-r| \leq \frac{|\xi|}{2}} d\xi\,  \frac{|\ga_{s,t}(\xi,r)|^2}{|\xi|^{2H-1}} \ .$$
On the one hand, it holds that
\begin{eqnarray*}
\int_{||\xi|-r| \geq \frac{|\xi|}{2} } d\xi\,  \frac{|\ga_{s,t}(\xi,r)|^2}{|\xi|^{2H-1}}& \lesssim & \frac{|t-s|^{2\ka}}{r^2}\int_{\{ |\xi| \leq \frac23 r\}\cup\{|\xi|\geq 2r\} }d\xi \frac{r^{2\ka}+|\xi|^{2\ka}}{|\xi|^{2H-1}||\xi|-r|^2}\\
&\lesssim & \frac{|t-s|^{2\ka}}{r^{2+2(H-\ka)}}\int_{\{ |\xi| \leq \frac23 \}\cup\{|\xi|\geq 2\}  }d\xi \frac{1+|\xi|^{2\ka}}{|\xi|^{2H-1}||\xi|-1|^2} \ \lesssim \ \frac{|t-s|^{2\ka}}{r^{2+2(H-\ka)}} \ .
\end{eqnarray*}
On the other hand, for any $\la\in [0,1]$, one has
\begin{eqnarray*}
\int_{||\xi|-r| \leq \frac{|\xi|}{2}} d\xi\,  \frac{|\ga_{s,t}(\xi,r)|^2}{|\xi|^{2H-1}}& \lesssim &\frac{|t-s|^{2\ka}}{r^2}\int_{\frac23 r \leq |\xi| \leq 2r }d\xi \frac{r^{2\ka}+|\xi|^{2\ka}}{|\xi|^{2H-1}||\xi|-r|^{2-2\la(1-\ka)}}\\
&\lesssim& \frac{|t-s|^{2\ka}}{r^{2+2(H-\ka)-2\la(1-\ka)}}\int_{\frac23  \leq |\xi| \leq 2 } \frac{d\xi}{|\xi|^{2H-1}||\xi|-1|^{2-2\la(1-\ka)}} \ ,
\end{eqnarray*}
and we get the conclusion by taking $\la = \frac{1+\varepsilon}{2(1-\ka)} \in [0,1]$.
\end{proof}

\

\subsection{Proof of Proposition \ref{prop:regu-psi}}
For the sake of clarity, we shall assume that $T\leq 1$ and set, for all $m,n \geq 1$, $\Psi_{n,m}\triangleq \Psi_m-\Psi_n$.

\smallskip

\noindent
\textbf{Step 1:} \change{Let us show that for all $m\geq n \geq 1$, $0\leq s<t\leq 1$ and $\varepsilon >0$ small enough, one has
\begin{equation}\label{step-1-order-one}
\begin{split}
&\int_{\R^d} dx \, \mathbb{E} \Big[ \big|\mathcal{F}^{-1} \big( \{1+|.|^{2}\}^{-\frac{\al}{2}} \mathcal{F}(\change{\rho}[\Psi_{n,m}(t,.)-\Psi_{n,m}(s,.)])\big)(x)\big|^{2p}\Big] \\
&\hspace{3cm}\lesssim 2^{-2n\varepsilon p} |t-s|^{2\varepsilon p} \ ,
\end{split}
\end{equation}
where the proportional constant only depends on $\rho$, $\al$ and $p$.}

\smallskip

\noindent
\change{Using the hypercontractivity property of Gaussian variables, we can first assert that}
\begin{align*}
&\mathbb{E} \Big[ \big|\mathcal{F}^{-1} \big( \{1+|.|^{2}\}^{-\frac{\al}{2}} \mathcal{F}(\change{\rho}[\Psi_{n,m}(t,.)-\Psi_{n,m}(s,.)])\big)(x)\big|^{2p}\Big]\\
&\leq c_p \, \mathbb{E} \Big[ \big|\mathcal{F}^{-1} \big( \{1+|.|^{2}\}^{-\frac{\al}{2}} \mathcal{F}(\change{\rho}[\Psi_{n,m}(t,.)-\Psi_{n,m}(s,.)])\big)(x)\big|^{2}\Big]^p \, ,
\end{align*}
where the constant $c_p$ only depends on $p$. Then write
\small
\begin{eqnarray*}
\lefteqn{\mathbb{E} \Big[ \big|\mathcal{F}^{-1} \big( \{1+|.|^{2}\}^{-\frac{\al}{2}} \mathcal{F}(\change{\rho}[\Psi_{n,m}(t,.)-\Psi_{n,m}(s,.)])\big)(x)\big|^{2}\Big]}\\
 &=&\int_{\R^d} d\la\int_{\R^d} dy\int_{\R^d} d\tilde{\la}\int_{\R^d} d\tilde{y}   \, e^{\imath \langle x,\la\rangle} \{1+|\la|^2\}^{-\frac{\al}{2}}e^{-\imath \langle \la,y\rangle} e^{-\imath \langle x,\tilde{\la}\rangle} \{1+|\tilde{\la}|^2\}^{-\frac{\al}{2}}\\
& &\hspace{1cm} e^{\imath \langle \tilde{\la},\tilde{y}\rangle}\change{\rho(y)\rho(\tilde{y})}\mathbb{E}\big[  \{\Psi_{n,m}(t,y)-\Psi_{n,m}(s,y)\}\overline{\{\Psi_{n,m}(t,\tilde{y})-\Psi_{n,m}(s,\tilde{y})\}}\big] \ .
\end{eqnarray*}
\normalsize
At this point, and with expression (\ref{expr-psi-n}) in mind, note that
\begin{align}
&\mathbb{E}\big[  \{\Psi_{n,m}(t,y)-\Psi_{n,m}(s,y)\}\overline{\{\Psi_{n,m}(t,\tilde{y})-\Psi_{n,m}(s,\tilde{y})\}}\big]\nonumber\\
&=c \int_{(\xi,\eta)\in \cd_{m,n}} d\xi d\eta \,  \frac{1}{|\xi|^{2H_0-1}} \prod_{i=1}^d \frac{1}{|\eta_i|^{2H_i-1}} |\ga_{s,t}(\xi,|\eta|)|^2 e^{\imath \langle \eta,y\rangle} e^{-\imath \langle \eta,\tilde{y}\rangle} \ ,\label{cov}
\end{align}
where $\cd_{m,n}\triangleq(\cb^1_m \times \cb^d_m) \backslash (\cb_n^1 \times \cb_n^d)$, $\cb_\ell^k\triangleq\{\la \in \R^k, \, |\la|\leq 2^\ell\}$, and accordingly
\begin{align*}
&\mathbb{E} \Big[ \big|\mathcal{F}^{-1} \big( \{1+|.|^{2}\}^{-\frac{\al}{2}} \mathcal{F}(\change{\rho} [\Psi_{n,m}(t,.)-\Psi_{n,m}(s,.)])\big)(x)\big|^2\Big]\\
&=c\int_{(\xi,\eta)\in \cd_{m,n}}d\xi d\eta\,  \frac{1}{|\xi|^{2H_0-1}} \prod_{i=1}^d \frac{1}{|\eta_i|^{2H_i-1}}\, |\ga_{s,t}(\xi,|\eta|)|^2\\
&\hspace{1cm}\change{\int_{\R^d}d\la \int_{\R^d} d\tilde{\la}\, e^{\imath \langle x,\la-\tilde{\la}\rangle}\hat{\rho}(\la-\eta) \hat{\rho}(\eta-\tilde{\la})\{1+|\la|^2\}^{-\frac{\al}{2}} \{1+|\tilde{\la}|^2\}^{-\frac{\al}{2}}} \, ,
\end{align*}
\change{which gives
\begin{align}
&\mathbb{E} \Big[ \big|\mathcal{F}^{-1} \big( \{1+|.|^{2}\}^{-\frac{\al}{2}} \mathcal{F}(\change{\rho} [\Psi_{n,m}(t,.)-\Psi_{n,m}(s,.)])\big)(x)\big|^2\Big]^p\nonumber\\
&=c\prod_{j=1}^p\int_{(\xi^j,\eta^j)\in \cd_{m,n}}d\xi^j d\eta^j\,  \frac{1}{|\xi^j|^{2H_0-1}} \prod_{i=1}^d \frac{1}{|\eta^j_i|^{2H_i-1}}\, |\ga_{s,t}(\xi^j,|\eta^j|)|^2\nonumber\\
&\hspace{0.2cm}\int_{\R^d}d\la^j \int_{\R^d} d\tilde{\la}^j\, e^{\imath \langle x,\la^j-\tilde{\la}^j\rangle}\hat{\rho}(\la^j-\eta^j) \hat{\rho}(\eta^j-\tilde{\la}^j)\{1+|\la^j|^2\}^{-\frac{\al}{2}} \{1+|\tilde{\la}^j|^2\}^{-\frac{\al}{2}} \, .\label{correc-1}
\end{align}
Now,
\begin{align}
&\int_{\R^d} dx \, \prod_{j=1}^p \int_{\R^d}d\la^j \int_{\R^d} d\tilde{\la}^j\, e^{\imath \langle x,\la^j-\tilde{\la}^j\rangle}\hat{\rho}(\la^j-\eta^j) \hat{\rho}(\eta^j-\tilde{\la}^j)\{1+|\la^j|^2\}^{-\frac{\al}{2}} \{1+|\tilde{\la}^j|^2\}^{-\frac{\al}{2}}\nonumber\\
&=\int_{\R^d} dx \, \prod_{j=1}^p \int_{\R^d}d\la^j \int_{\R^d} d\tilde{\la}^j\, e^{\imath \langle x,\la^j-\tilde{\la}^j\rangle}\hat{\rho}(\la^j) \hat{\rho}(-\tilde{\la}^j)\{1+|\eta^j+\la^j|^2\}^{-\frac{\al}{2}} \{1+|\eta^j+\tilde{\la}^j|^2\}^{-\frac{\al}{2}}\nonumber\\
&=\prod_{j=1}^{p-1} \int_{\R^d}d\la^j \int_{\R^d} d\tilde{\la}^j\, \hat{\rho}(\la^j) \hat{\rho}(-\tilde{\la}^j)\{1+|\eta^j+\la^j|^2\}^{-\frac{\al}{2}} \{1+|\eta^j+\tilde{\la}^j|^2\}^{-\frac{\al}{2}}\nonumber\\
&\hspace{3cm}\int_{\R^d} d\la^p \, \hat{\rho}(\la^p)\hat{\rho}\Big(\sum\nolimits_{k=1}^{p-1}(\tilde{\la}^k-\la^k)+\la^p\Big)\{1+|\eta^p+\la^p|^2\}^{-\frac{\al}{2}}\nonumber\\
& \hspace{6cm}\Big\{1+\Big|\eta^p+\sum\nolimits_{k=1}^{p-1}(\tilde{\la}^k-\la^k)+\la^p\Big|^2\Big\}^{-\frac{\al}{2}} \, .\label{correc-2}
\end{align}
The absolute value of this product can in fact be bounded by
$$c\prod_{i=1}^p \{1+|\eta^j|^2\}^{-\al} \, ,$$
for some constant $c>0$, due to
\begin{align}
&\big| \hat{\rho}(\la)\{1+|\eta+\la|^2\}^{-\frac{\al}{2}} \big|\nonumber\\
&=\big| \hat{\rho}(\la)\{1+|\eta+\la|^2\}^{-\frac{\al}{2}} \big| \1_{\{|\la|>\frac12 |\eta|\}}+\big| \hat{\rho}(\la)\{1+|\eta+\la|^2\}^{-\frac{\al}{2}} \big| \1_{\{|\la|<\frac12 |\eta|\}}\nonumber\\
&\leq c_\al \Big[ |\hat{\rho}(\la)| \1_{\{|\la|>\frac12 |\eta|\}}+|\hat{\rho}(\la)|\{1+|\eta|^2\}^{-\frac{\al}{2}}  \1_{\{|\la|<\frac12 |\eta|\}}\Big]\nonumber\\
&\leq c_{\al,\rho,\ka}\Big[ \{1+|\la|^2\}^{-\ka} \{1+|\la|^2\}^{-\frac{\al}{2}} \1_{\{|\la|>\frac12 |\eta|\}}+\{1+|\la|^2\}^{-\ka}\{1+|\eta|^2\}^{-\frac{\al}{2}}  \1_{\{|\la|<\frac12 |\eta|\}}\Big]\nonumber\\
&\leq c_{\al,\rho,\ka} \{1+|\la|^2\}^{-\ka} \{1+|\eta|^2\}^{-\frac{\al}{2}} \, ,\label{correc-3}
\end{align}
for all $\la,\eta\in \R^d$ and $\ka >0$.}

\smallskip

\change{Going back to (\ref{correc-1}), we get that
\begin{align}
&\mathbb{E} \Big[ \big|\mathcal{F}^{-1} \big( \{1+|.|^{2}\}^{-\frac{\al}{2}} \mathcal{F}(\rho[\Psi_{n,m}(t,.)-\Psi_{n,m}(s,.)])\big)(x)\big|^2\Big]^p\nonumber\\
&\lesssim \bigg(\int_{(\xi,\eta)\in \cd_{m,n}}d\xi d\eta\,  \frac{1}{|\xi|^{2H_0-1}} \prod_{i=1}^d \frac{1}{|\eta_i|^{2H_i-1}}\{1+|\eta|^{2}\}^{-\al}\, |\ga_{s,t}(\xi,|\eta|)|^2\bigg)^p\nonumber\\
&\lesssim \bigg(\int_{2^n\leq |\xi|\leq 2^m}d\xi\int_{|\eta|\leq 2^m}d\eta\,  \frac{1}{|\xi|^{2H_0-1}} \prod_{i=1}^d \frac{1}{|\eta_i|^{2H_i-1}}\{1+|\eta|^{2}\}^{-\al}\, |\ga_{s,t}(\xi,|\eta|)|^2\bigg)^p\nonumber\\
&+\bigg(\int_{|\xi|\leq 2^m}d\xi\int_{2^n \leq |\eta|\leq 2^m}d\eta\,  \frac{1}{|\xi|^{2H_0-1}} \prod_{i=1}^d \frac{1}{|\eta_i|^{2H_i-1}}\{1+|\eta|^{2}\}^{-\al}\, |\ga_{s,t}(\xi,|\eta|)|^2\bigg)^p  \nonumber\\
& \triangleq \big( \text{I}_{m,n}(s,t)\big)^p+\big( \text{II}_{m,n}(s,t)\big)^p \ .\label{decompo-proof-ordre-un}
\end{align}}
Let us focus on the estimation of $\text{I}_{m,n}(s,t)$ (the treatment of $\text{II}_{m,n}(s,t)$ can be done along similar arguments). Using an elementary spherical change-of-variable for the $\eta_i$-coordinates, we get that for any $0<\varepsilon<H_0$,
\begin{align*}
&\text{I}_{m,n}(s,t) \leq 2^{-2n\varepsilon} \int_{\R}\frac{d\xi}{|\xi|^{2H_0-2\varepsilon-1}}\int_0^\infty dr \, \frac{\{1+r^2\}^{-\al}}{r^{2(H_1+\ldots+H_d)-2d+1}}\, |\ga_{s,t}(\xi,r)|^2  \\
&\quad \int_{[0,2\pi]^{d-1}} d\theta_1 \cdots d\theta_{d-1}\prod_{i=1}^{d-1}\frac{1}{|\cos(\theta_i)|^{2H_i-1}|\sin(\theta_i)|^{2(H_{i+1}+\ldots+H_d)-2d+2i+1}} \ ,
\end{align*}
and since $\max(2H_i-1,2(H_{i+1}+\ldots+H_d)-2d+2i+1)<1$ for every $i\in \{1,\ldots,d-1\}$, this yields
$$\text{I}_{m,n}(s,t) \lesssim 2^{-2n\varepsilon}\int_{\R}d\xi \int_0^\infty dr\, \frac{1}{|\xi|^{2H_0-2\varepsilon-1}}\, \frac{\{1+r^2\}^{-\al}}{r^{2(H_1+\ldots+H_d)-2d+1}}\, |\ga_{s,t}(\xi,r)|^2 \ .$$
Now, by applying Corollary \ref{coro:tech}, we can assert that for all $0<\varepsilon <\min(H_0,\frac12)$ and $0<\ka<\min(H_0-\varepsilon,\frac12-\varepsilon)$,
\begin{align}
&\int_{\R}d\xi \int_0^\infty dr\, \frac{1}{|\xi|^{2H_0-2\varepsilon-1}}\, \frac{\{1+r^2\}^{-\al}}{r^{2(H_1+\ldots+H_d)-2d+1}}\, |\ga_{s,t}(\xi,r)|^2 \nonumber \\
& \lesssim  |t-s|^{2\ka}\bigg[\int_0^1 \frac{dr}{r^{2(H_1+\ldots+H_d)-2d+1}}+\int_1^\infty dr\, \frac{1}{r^{2\al+2(H_0+\ldots+H_d)-2d+2-2\ka-4\varepsilon}} \bigg]\label{correct-brack} \ .
\end{align}
The conclusion is straightforward: the two integrals involved in (\ref{correct-brack}) are indeed finite as soon as 
$$2\varepsilon+\ka < \al- \Big[ d-\frac12-\sum_{i=0}^d H_i\Big] \ .$$

\

\noindent
\textbf{Step 2:} \change{We have thus shown that
$$\mathbb{E} \Big[ \big\|\rho \Psi_{n,m}(t,.)-\rho\Psi_{n,m}(s,.) \big\|_{\cw^{-\al,2p}(\change{\R^d})}^{2p} \Big]\lesssim 2^{-2n\varepsilon p} |t-s|^{2\varepsilon p}\ ,$$}
and we can now conclude by applying the classical Garsia-Rodemich-Rumsey estimate: for any $\varepsilon_0>0$,
\begin{eqnarray*}
\lefteqn{\mathbb{E} \Big[ \cn \big[\change{\rho}\Psi_{n,m};\cac^{\varepsilon_0}([0,T];\cw^{-\al,2p}(\change{\R^d}))\big]^{2p} \Big]}\\
 &\lesssim &\iint_{[0,1]^2} ds dt \frac{\mathbb{E} \Big[ \big\| \change{\rho}\Psi_{n,m}(t,.)-\change{\rho}\Psi_{n,m}(s,.) \big\|_{\cw^{-\al,2p}(\change{\R^d})}^{2p} \Big]}{|t-s|^{2\varepsilon_0 p+2}}\\
&\lesssim & 2^{-2n\varepsilon p}\iint_{[0,1]^2}\frac{ds dt}{|t-s|^{-2(\varepsilon-\varepsilon_0) p+2}} \ ,
\end{eqnarray*}
noting that the latter integral is finite for all $0<\varepsilon_0<\varepsilon$ and $p$ large enough.

\

\subsection{Proof of Proposition \ref{prop:regu-psi-order-two}}\label{subsec:proof-conv-order-two}

Due to condition \eqref{constraint-h-i}, we can (and will) assume in the sequel that $\al <\frac14$, which will be of importance in our estimates (see (\ref{importance-alpha-1-4})). Also, for the sake of clarity, we shall again assume that $T\leq 1$. Finally, let us set, for all $m,n \geq 1$ and $0\leq s,t\leq 1$, $\Psi_{n,m}\triangleq \Psi_m-\Psi_n$, $\widehat{\mathbf{\Psi}}^{\mathbf{2}}_{n,m}\triangleq \widehat{\mathbf{\Psi}}^{\mathbf{2}}_m-\widehat{\mathbf{\Psi}}^{\mathbf{2}}_n$ and $f(s,t;.)\triangleq f(t,.)-f(s,.)$ for $f\in \{\Psi_n,\Psi_{n,m},\widehat{\mathbf{\Psi}}^{\mathbf{2}}_{n},\widehat{\mathbf{\Psi}}^{\mathbf{2}}_{n,m}\}$.

\smallskip

Just as in \cite{gubinelli-koch-oh}, the success of the renormalization procedure essentially lies in the following elementary property, which can be readily derived from the classical Wick formula:
\begin{lemma}\label{lem:renorm}
For all $m,n\geq 1$, $s,t\geq 0$ and $y,\tilde{y}\in \R$, it holds that
\begin{equation*}
\mathbb{E}\big[\widehat{\mathbf{\Psi}}^{\mathbf{2}}_m(t,y)\ \overline{\widehat{\mathbf{\Psi}}^{\mathbf{2}}_n(s,\tilde{y})} \big]=2 \mathbb{E}\big[ \Psi_m(t,y) \overline{\Psi_n(s,\tilde{y})}\big]^2 \ .
\end{equation*}
\end{lemma}

\

We can now turn to the proof of Proposition \ref{prop:regu-psi-order-two}, that we present as a two-step procedure (just as the proof of Proposition \ref{prop:regu-psi}). 

\smallskip

\noindent
\textbf{Step 1:} Let us show that for all $m\geq n \geq 1$, $0\leq s\leq t\leq 1$ and $\varepsilon >0$ small enough, one has
\change{\begin{equation}\label{mom-ordre-deux}
\int_{\R^d} dx \, \mathbb{E} \Big[ \big|\mathcal{F}^{-1} \big( \{1+|.|^{2}\}^{-\al} \mathcal{F}(\rho^2\widehat{\mathbf{\Psi}}^{\mathbf{2}}_{n,m}(s,t;.))\big)(x)\big|^{2p}\Big] \lesssim 2^{-2n\varepsilon p}|t-s|^{2\varepsilon p} \ ,
\end{equation}
where the proportional constant only depends on $\rho$, $\al$ and $p$.}

\smallskip

\change{Using the hypercontractivity property of Wiener chaoses, we can first assert that
\begin{align*}
&\mathbb{E} \Big[ \big|\mathcal{F}^{-1} \big( \{1+|.|^{2}\}^{-\al} \mathcal{F}(\rho^2\widehat{\mathbf{\Psi}}^{\mathbf{2}}_{n,m}(s,t;.))\big)(x)\big|^{2p}\Big]\\
&\leq c_p \, \mathbb{E} \Big[ \big|\mathcal{F}^{-1} \big( \{1+|.|^{2}\}^{-\al} \mathcal{F}(\rho^2\widehat{\mathbf{\Psi}}^{\mathbf{2}}_{n,m}(s,t;.))\big)(x)\big|^{2}\Big]^p \, ,
\end{align*}
where the constant $c_p$ only depends on $p$.} Then write
\begin{eqnarray*}
\lefteqn{\mathbb{E} \Big[ \big|\mathcal{F}^{-1} \big( \{1+|.|^{2}\}^{-\al} \mathcal{F}(\rho^2\widehat{\mathbf{\Psi}}^{\mathbf{2}}_{n,m}(s,t;.))\big)(x)\big|^{2}\Big]}\\
   \\& &=\int_{\R^d} d\la\int_{\R^d} dy\int_{\R^d} d\tilde{\la}\int_{\R^d} d\tilde{y}\, e^{\imath \langle x,\la\rangle} \{1+|\la|^2\}^{-\al}e^{-\imath \langle \la,y\rangle} e^{-\imath \langle x,\tilde{\la}\rangle} \\
	& & \hspace{1.5cm}\{1+|\tilde{\la}|^2\}^{-\al}e^{\imath \langle \tilde{\la},\tilde{y}\rangle}\rho^2(y)\rho^2(\tilde{y})\mathbb{E}\big[  \widehat{\mathbf{\Psi}}^{\mathbf{2}}_{n,m}(s,t;y)\overline{\widehat{\mathbf{\Psi}}^{\mathbf{2}}_{n,m}(s,t;\tilde{y})}\big] \ ,
\end{eqnarray*}
and, using Lemma \ref{lem:renorm}, we can check that
\begin{eqnarray*}
\frac12\mathbb{E}\big[  \widehat{\mathbf{\Psi}}^{\mathbf{2}}_{n,m}(s,t;y)\overline{\widehat{\mathbf{\Psi}}^{\mathbf{2}}_{n,m}(s,t;\tilde{y})}\big]&=&\mathbb{E}\big[  \Psi_{n,m}(t,y)\overline{\Psi_{n,m}(s,t;\tilde{y})}\big] \mathbb{E}\big[  \Psi_m(t,y)\overline{\{\Psi_m+\Psi_n\}(t,\tilde{y})}\big]\\
& &+\mathbb{E}\big[  \Psi_{n,m}(t,y)\overline{\Psi_{n,m}(s,\tilde{y})}\big] \mathbb{E}\big[  \Psi_{m}(t,y)\overline{\{\Psi_m+\Psi_n\}(s,t;\tilde{y})}\big] \\
& &+\mathbb{E}\big[  \Psi_{n}(t,y)\overline{\Psi_{n,m}(s,t;\tilde{y})}\big] \mathbb{E}\big[  \Psi_{n,m}(t,y)\overline{\{\Psi_m+\Psi_n\}(t,\tilde{y})}\big]\\
& &+\mathbb{E}\big[  \Psi_{m}(t,y)\overline{\Psi_{n,m}(s,\tilde{y})}\big] \mathbb{E}\big[  \Psi_{n,m}(t,y)\overline{\{\Psi_m+\Psi_n\}(s,t;\tilde{y})}\big] \\
& & +\mathbb{E}\big[  \Psi_{n,m}(s,y)\overline{\Psi_{n,m}(t,s;\tilde{y})}\big] \mathbb{E}\big[  \Psi_m(s,y)\overline{\{\Psi_m+\Psi_n\}(s,\tilde{y})}\big]\\
& &+\mathbb{E}\big[  \Psi_{n,m}(s,y)\overline{\Psi_{n,m}(t,\tilde{y})}\big] \mathbb{E}\big[  \Psi_{m}(s,y)\overline{\{\Psi_m+\Psi_n\}(t,s;\tilde{y})}\big] \\
& &+\mathbb{E}\big[  \Psi_{n}(s,y)\overline{\Psi_{n,m}(t,s;\tilde{y})}\big] \mathbb{E}\big[  \Psi_{n,m}(s,y)\overline{\{\Psi_m+\Psi_n\}(s,\tilde{y})}\big]\\
& &+\mathbb{E}\big[  \Psi_{m}(s,y)\overline{\Psi_{n,m}(t,\tilde{y})}\big] \mathbb{E}\big[  \Psi_{n,m}(s,y)\overline{\{\Psi_m+\Psi_n\}(t,s;\tilde{y})}\big] \\
& \triangleq &  \sum\nolimits_{i=1,\ldots,8} \text{A}^i_{n,m}(s,t;y,\tilde{y}) \ .
\end{eqnarray*}
It turns out that the eight terms derived from $\text{A}^i_{m,n}(s,t;y,\tilde{y})$ ($i\in \{1,\ldots,8\}$) can be handled with the same arguments, and therefore we will only focus on the treatment of $\text{A}^1_{m,n}(s,t;y,\tilde{y})$. In fact, just as in (\ref{cov}), one has
\begin{align*}
&\mathbb{E}\big[  \Psi_{n,m}(t,y)\overline{\Psi_{n,m}(s,t;\tilde{y})}\big]=c \int_{(\xi,\eta)\in \cd_{m,n}} d\xi d\eta \,  \frac{1}{|\xi|^{2H_0-1}} \prod_{i=1}^d \frac{1}{|\eta_i|^{2H_i-1}}  e^{\imath \langle \eta,y\rangle} e^{-\imath \langle \eta,\tilde{y}\rangle}\ga_t(\xi,|\eta|) \overline{\ga_{s,t}(\xi,|\eta|)}
\end{align*}
and
\begin{align*}
&\mathbb{E}\big[  \Psi_m(t,y)\overline{\{\Psi_m+\Psi_n\}(t,\tilde{y})}\big]=c \int_{\R \times \R^d } d\xiti d\etati \,  \frac{1}{|\xiti|^{2H_0-1}} \prod_{i=1}^d \frac{1}{|\etati_i|^{2H_i-1}} |\ga_t(\xiti,|\etati|)|^2 e^{\imath \langle \etati,y\rangle} e^{-\imath \langle \etati,\tilde{y}\rangle}\\
&\hspace{9cm} \big\{ \mathbf{1}_{(\xiti,\etati)\in \cb_m^1 \times \cb_m^d}+ \mathbf{1}_{(\xiti,\etati)\in \cb_n^1 \times \cb_n^d} \big\} \ ,
\end{align*}
where $\cd_{m,n}\triangleq(\cb^1_m \times \cb^d_m) \backslash (\cb_n^1 \times \cb_n^d)$ and $\cb_\ell^k\triangleq\{\la \in \R^k, \, |\la|\leq 2^\ell\}$. \change{Besides, one has obviously
\begin{align*}
&\int_{\R^d} d\la\int_{\R^d} dy\int_{\R^d} d\tilde{\la}\int_{\R^d} d\tilde{y}   \, e^{\imath \langle x,\la\rangle} \{1+|\la|^2\}^{-\al}e^{-\imath \langle \la,y\rangle} e^{-\imath \langle x,\tilde{\la}\rangle}\\
&\hspace{2cm}\{1+|\tilde{\la}|^2\}^{-\al}e^{\imath \langle \tilde{\la},\tilde{y}\rangle} \rho^2(y) \rho^2(\tilde{y})e^{\imath \langle \eta,y\rangle} e^{-\imath \langle \eta,\tilde{y}\rangle}e^{\imath \langle \etati,y\rangle} e^{-\imath \langle \etati,\tilde{y}\rangle}\\
&=\int_{\R^d} d\la\int_{\R^d} d\tilde{\la}\, e^{\imath \langle x, \la-\tilde{\la}\rangle} \{1+|\la|^2\}^{-\al}\{1+|\tilde{\la}|^2\}^{-\al} \cf\big(\rho^2\big)(\la-\eta-\etati)\cf\big(\rho^2\big)(\eta+\etati-\tilde{\la}) \, .
\end{align*}
Using the same arguments as in the proof of Proposition \ref{prop:regu-psi} (see (\ref{correc-1})-(\ref{correc-3})), we end up with}
\begin{align*}
&\int_{\R^d} dx \, \bigg|\int_{\R^d} d\la\int_{\R^d} dy\int_{\R^d} d\tilde{\la}\int_{\R^d} d\tilde{y}   \, e^{\imath \langle x,\la\rangle} \{1+|\la|^2\}^{-\al}e^{-\imath \langle \la,y\rangle} e^{-\imath \langle x,\tilde{\la}\rangle}\\
&\hspace{6cm} \{1+|\tilde{\la}|^2\}^{-\al}e^{\imath \langle \tilde{\la},\tilde{y}\rangle} \text{A}^1_{m,n}(s,t;y,\tilde{y})\bigg|^p\\
&\lesssim\bigg( \int_{(\xi,\eta)\in \cd_{m,n}} d\xi d\eta \,  \frac{1}{|\xi|^{2H_0-1}} \prod_{i=1}^d \frac{1}{|\eta_i|^{2H_i-1}} \big| \ga_t(\xi,|\eta|)\big| \big| \ga_{s,t}(\xi,|\eta|)\big| \\
& \hspace{3.5cm} \int_{\R \times \R^d } d\xiti d\etati \,  \frac{1}{|\xiti|^{2H_0-1}} \prod_{i=1}^d \frac{1}{|\etati_i|^{2H_i-1}} |\ga_t(\xiti,|\etati|)|^2\\
& \hspace{4cm} \big\{ \mathbf{1}_{(\xiti,\etati)\in \cb_m^1 \times \cb_m^d}+ \mathbf{1}_{(\xiti,\etati)\in \cb_n^1 \times \cb_n^d} \big\} \{1+|\eta-\etati|^2\}^{-2\al}\bigg)^p\\
&\lesssim \bigg(\int_{|\xi|\geq 2^n} d\xi\int_{\R^d} d\eta \,  \frac{1}{|\xi|^{2H_0-1}} \prod_{i=1}^d \frac{1}{|\eta_i|^{2H_i-1}} |\ga_t(\xi,|\eta|)| |\ga_{s,t}(\xi,|\eta|)|\\
&\hspace{2cm} \int_{\R \times \R^d } d\xiti d\etati \,  \frac{1}{|\xiti|^{2H_0-1}} \prod_{i=1}^d \frac{1}{|\etati_i|^{2H_i-1}} |\ga_t(\xiti,|\etati|)|^2  \{1+|\eta-\etati|^2\}^{-2\al}\bigg)^p \\
&\ \ \ +\bigg(\int_{\R} d\xi \int_{|\eta| \geq 2^n} d\eta \,  \frac{1}{|\xi|^{2H_0-1}} \prod_{i=1}^d \frac{1}{|\eta_i|^{2H_i-1}} |\ga_t(\xi,|\eta|)| |\ga_{s,t}(\xi,|\eta|)|\\
&\hspace{2cm} \int_{\R \times \R^d } d\xiti d\etati \,  \frac{1}{|\xiti|^{2H_0-1}} \prod_{i=1}^d \frac{1}{|\etati_i|^{2H_i-1}} |\ga_t(\xiti,|\etati|)|^2  \{1+|\eta-\etati|^2\}^{-2\al}\bigg)^p \\
& \triangleq  \big(\text{I}_{n}(s,t)\big)^p+\big(\text{II}_{n}(s,t)\big)^p \ .
\end{align*}
As in the proof of Proposition \ref{prop:regu-psi}, we will restrict our attention to $\text{I}_{n}(s,t)$. For $0<\varepsilon <H_0$, one has 
\small
\begin{align}
\text{I}_n(s,t) \lesssim & \ 2^{-2n\varepsilon}\int_{\R \times \R^d} d\xi d\eta \int_{\R \times \R^d } d\xiti d\etati \, \{1+|\eta-\etati|^2\}^{-2\al}\nonumber\\
& \hspace{0.5cm}\frac{1}{|\xi|^{2(H_0-\varepsilon)-1}} \prod_{i=1}^d \frac{1}{|\eta_i|^{2H_i-1}} |\ga_t(\xi,|\eta|)| |\ga_{s,t}(\xi,|\eta|)|  \frac{1}{|\xiti|^{2H_0-1}} \prod_{i=1}^d \frac{1}{|\etati_i|^{2H_i-1}} |\ga_t(\xiti,|\etati|)|^2 \label{i-n-s-t-0} \\
\lesssim& \ 2^{-2n\varepsilon}\int_{\R \times \R^d} d\xi d\eta \int_{\R \times \R^d } d\xiti d\etati \, \{1+||\eta|-|\etati||^2\}^{-2\al} \nonumber \\
& \hspace{0.5cm}  \frac{1}{|\xi|^{2(H_0-\varepsilon)-1}} \prod_{i=1}^d \frac{1}{|\eta_i|^{2H_i-1}} |\ga_t(\xi,|\eta|)| |\ga_{s,t}(\xi,|\eta|)|  \frac{1}{|\xiti|^{2H_0-1}} \prod_{i=1}^d \frac{1}{|\etati_i|^{2H_i-1}} |\ga_t(\xiti,|\etati|)|^2 \ .\label{i-n-s-t}
\end{align}
\normalsize
Now let us split the integration domain as $(\R \times \R^d)^2 = D_1\cup D_2$, with
$$D_1\triangleq \{(\xi,\eta,\xiti,\etati): \  \frac{|\eta|}{2} < |\etati| < \frac{3 |\eta|}{2} \}$$
and
$$D_2\triangleq \{(\xi,\eta,\xiti,\etati): \  0<|\etati| < \frac{|\eta|}{2} \ \text{or} \ |\etati| > \frac{3|\eta|}{2}\}   \ . $$ 
For $(\xi,\eta,\xiti,\etati) \in D_2$, one has $||\eta|-|\etati|| > \max \big( \frac{|\eta|}{2} , \frac{|\etati|}{3}\big)$, and so
\begin{align}
&\int_{D_2} \frac{d\xi d\eta d\xiti d\etati}{\{1+||\eta|-|\etati||^2\}^{2\al}}\frac{1}{|\xi|^{2(H_0-\varepsilon)-1}} \prod_{i=1}^d \frac{1}{|\eta_i|^{2H_i-1}} |\ga_t(\xi,|\eta|)| |\ga_{s,t}(\xi,|\eta|)|\frac{1}{|\xiti|^{2H_0-1}} \prod_{i=1}^d \frac{1}{|\etati_i|^{2H_i-1}} |\ga_t(\xiti,|\etati|)|^2\nonumber\\
&\lesssim \bigg(\int_{\R \times \R^d} \frac{d\xi d\eta }{\{1+|\eta|^2\}^{\al}}\frac{1}{|\xi|^{2(H_0-\varepsilon)-1}} \prod_{i=1}^d \frac{1}{|\eta_i|^{2H_i-1}} |\ga_t(\xi,|\eta|)| |\ga_{s,t}(\xi,|\eta|)|\bigg)\nonumber\\
&\hspace{5cm} \bigg( \int_{\R \times \R^d} \frac{d\xiti d\etati }{\{1+|\etati|^2\}^{\al}} \frac{1}{|\xiti|^{2H_0-1}} \prod_{i=1}^d \frac{1}{|\etati_i|^{2H_i-1}} |\ga_t(\xiti,|\etati|)|^2 \bigg) \nonumber\\
&\lesssim \bigg(\int_{\R \times \R^d} \frac{d\xi d\eta }{\{1+|\eta|^2\}^{\al}}\frac{1}{|\xi|^{2(H_0-\varepsilon)-1}} \prod_{i=1}^d \frac{1}{|\eta_i|^{2H_i-1}} |\ga_t(\xi,|\eta|)|^2\bigg)^{1/2}\nonumber\\
&\hspace{2cm}\bigg(\int_{\R \times \R^d} \frac{d\xi d\eta }{\{1+|\eta|^2\}^{\al}}\frac{1}{|\xi|^{2(H_0-\varepsilon)-1}} \prod_{i=1}^d \frac{1}{|\eta_i|^{2H_i-1}} |\ga_{s,t}(\xi,|\eta|)|^2\bigg)^{1/2}\nonumber\\
&\hspace{5cm} \bigg( \int_{\R \times \R^d} \frac{d\xiti d\etati }{\{1+|\etati|^2\}^{\al}} \frac{1}{|\xiti|^{2H_0-1}} \prod_{i=1}^d \frac{1}{|\etati_i|^{2H_i-1}} |\ga_t(\xiti,|\etati|)|^2 \bigg) \ .\label{domain-d-1}
\end{align}
At this point, observe that we are exactly in the same position as in the proof of Proposition \ref{prop:regu-psi} (see (\ref{decompo-proof-ordre-un})), and so we can rely on the same arguments to assert that for $\varepsilon >0$ small enough, the above integral (over $D_2$) is indeed bounded by $c|t-s|^\varepsilon$, for some finite constant $c$.

\smallskip

\noindent
In order to deal with the integral over the domain $D_1$, observe first that
\begin{align*}
&\int_{\frac{|\eta|}{2} < |\etati| < \frac{3|\eta|}{2}} \frac{d\etati}{\{1+||\eta|-|\etati||^2\}^{2\al}}  |\ga_t(\xiti,|\etati|)|^2\prod_{i=1}^d \frac{1}{|\etati_i|^{2H_i-1}} \\
&=|\eta|^{-2(H_1+\ldots+H_d)+2d}  \int_{\frac{1}{2} < |\etati| < \frac32} \frac{d\etati}{\{1+|\eta|^2(1-|\etati|)^2\}^{2\al}}  |\ga_t(\xiti,|\eta| |\etati|)|^2\prod_{i=1}^d \frac{1}{|\etati_i|^{2H_i-1}}\\
&\lesssim |\eta|^{-2(H_1+\ldots+H_d)+2d} \int_{\frac12}^{\frac32} \frac{dr}{\{1+|\eta|^2(1-r)^2\}^{2\al}}  |\ga_t(\xiti,|\eta| r)|^2 \ ,
\end{align*}
and so 
\small
\begin{align*}
&\int_{D_1} \frac{d\xi d\eta d\xiti d\etati}{\{1+||\eta|-|\etati||^2\}^{2\al}}\frac{1}{|\xi|^{2(H_0-\varepsilon)-1}} \prod_{i=1}^d \frac{1}{|\eta_i|^{2H_i-1}} |\ga_t(\xi,|\eta|)| |\ga_{s,t}(\xi,|\eta|)| \frac{1}{|\xiti|^{2H_0-1}} \prod_{i=1}^d \frac{1}{|\etati_i|^{2H_i-1}} |\ga_t(\xiti,|\etati|)|^2\\
&\lesssim \int_{\R^d} \frac{d\eta}{|\eta|^{2(H_1+\ldots+H_d)-2d}}\prod_{i=1}^d \frac{1}{|\eta_i|^{2H_i-1}} \int_{\frac12}^{\frac32} \frac{dr}{\{1+|\eta|^2(1-r)^2\}^{2\al}}\\
&\hspace{7cm}\int_{\R} d\xi \, \frac{|\ga_t(\xi,|\eta|)| |\ga_{s,t}(\xi,|\eta|)|}{|\xi|^{2(H_0-\varepsilon)-1}} \int_{\R} d\xiti \, \frac{|\ga_t(\xiti,|\eta| r)|^2}{|\xiti|^{2H_0-1}}\\
&\lesssim \int_{\R^d} \frac{d\eta}{|\eta|^{2(H_1+\ldots+H_d)-2d}}\prod_{i=1}^d \frac{1}{|\eta_i|^{2H_i-1}} \int_{\frac12}^{\frac32} \frac{dr}{\{1+|\eta|^2(1-r)^2\}^{2\al}}\\
&\hspace{4cm}\bigg(\int_{\R} d\xi \, \frac{|\ga_t(\xi,|\eta|)|^2 }{|\xi|^{2(H_0-\varepsilon)-1}}\bigg)^{1/2}\bigg(\int_{\R} d\xi \, \frac{ |\ga_{s,t}(\xi,|\eta|)|^2}{|\xi|^{2(H_0-\varepsilon)-1}}\bigg)^{1/2} \int_{\R} d\xiti \, \frac{|\ga_t(\xiti,|\eta| r)|^2}{|\xiti|^{2H_0-1}}\\
&\lesssim \int_0^\infty \frac{d\rho}{\rho^{4(H_1+\ldots+H_d)-4d+1}} \int_{\frac12}^{\frac32} \frac{dr}{\{1+\rho^2(1-r)^2\}^{2\al}}\\
&\hspace{4cm}\bigg(\int_{\R} d\xi \, \frac{|\ga_t(\xi,\rho)|^2 }{|\xi|^{2(H_0-\varepsilon)-1}}\bigg)^{1/2}\bigg(\int_{\R} d\xi \, \frac{ |\ga_{s,t}(\xi,\rho)|^2}{|\xi|^{2(H_0-\varepsilon)-1}}\bigg)^{1/2} \int_{\R} d\xiti \, \frac{|\ga_t(\xiti,\rho  r)|^2}{|\xiti|^{2H_0-1}}\\
&\lesssim |t-s|^\ka \bigg[\int_0^1 \frac{d\rho}{\rho^{4(H_1+\ldots+H_d)-4d+1}}+\int_1^\infty \frac{d\rho}{\rho^{4(H_0+\ldots+H_d)-4d+3-8\varepsilon-\ka}}\int_{\frac12}^{\frac32}  \frac{dr}{\{1+\rho^2(1-r)^2\}^{2\al}}\bigg] \
\end{align*}
\normalsize
for all $0<\varepsilon < \min(H_0,\frac12)$ and $0<\ka< \min(H_0-\varepsilon,\frac12-\varepsilon)$, where we have used Corollary \ref{coro:tech} to derive the last inequality. Finally, since $\al <\frac14$, it is readily checked that for all $0<\varepsilon < \min(H_0,\frac12)$ and $0<\ka< \min(H_0-\varepsilon,\frac12-\varepsilon)$ such that $8\varepsilon +\ka < 4\Big(\al- \Big[ d-\frac12-\sum_{i=0}^d H_i\Big]\Big)$,
\begin{align}
&\int_1^\infty \frac{d\rho}{\rho^{4(H_0+\ldots+H_d)-4d+3-8\varepsilon-\ka}}\int_{\frac12}^{\frac32}  \frac{dr}{\{1+\rho^2(1-r)^2\}^{2\al}}\nonumber\\
&\leq \int_1^\infty \frac{d\rho}{\rho^{4\al+4(H_0+\ldots+H_d)-4d+3-8\varepsilon-\ka}}\int_{\frac12}^{\frac32}  \frac{dr}{(1-r)^{4\al}} \ < \ \infty \ . \label{importance-alpha-1-4}
\end{align}
Going back to (\ref{i-n-s-t}), we have thus shown (\ref{mom-ordre-deux}).

\

\noindent
\textbf{Step 2:} Once endowed with estimate (\ref{mom-ordre-deux}), \change{we can of course use the same arguments as in Step 2 of the proof of Proposition \ref{prop:regu-psi} to obtain that}, for $0<\varepsilon_0<\varepsilon$ and $p$ large enough,
$$\mathbb{E} \Big[ \cn \big[\widehat{\mathbf{\Psi}}^{\mathbf{2}}_{n,m};\cac^{\varepsilon_0}([0,T];\cw^{-2\al,2p}(D))\big]^{2p} \Big] \lesssim 2^{-n\varepsilon p} \ ,$$
which completes the proof our assertion.

\

\

\subsection{Estimation of the renormalization constant}

Let us conclude this section with the asymptotic analysis of the renormalization constant $\si_n(t) \ \triangleq \ \mathbb{E}\big[ \Psi_n(t,x)^2\big]$ at the core of the above renormalization procedure. In other words, our aim here is to show (\ref{estim-cstt}). To this end, fix $d\geq 2$ and $(H_0,\ldots,H_d)\in (0,1)^{d+1}$ such that 
$$
d-\frac34 < \sum_{i=0}^d H_i \leq  d-\frac12 \ ,
$$
and, with expression (\ref{expr-psi-n}) in mind, write the renormalization constant as
\begin{eqnarray*}
\si_n(t)  &=& \mathbb{E}\big[ \Psi_n(t,x)^2\big]\\
&=&c \int_{|\xi|\leq 2^n}  \frac{d\xi}{|\xi|^{2H_0-1}}\int_{|\eta|\leq 2^n} \prod_{i=1}^d \frac{d\eta_i}{|\eta_i|^{2H_i-1}} \, |\ga_t(\xi,|\eta|)|^2 \\
&=&c \int_0^{2^n} \frac{dr}{r^{2(H_1+\ldots+H_d)-2d+1}}\int_{|\xi|\leq 2^n}  \frac{d\xi}{|\xi|^{2H_0-1}} \, |\ga_{t}(\xi,r)|^2 \ .
\end{eqnarray*}
The asymptotic estimate (\ref{estim-cstt}) is now a straightforward consequence of the following technical result (take $\al \triangleq 2H_0 \in (0,2)$ and $\ka \triangleq 2(d-\frac12-\sum_{i=0}^d H_i ) \in [0,1)$):
\begin{proposition}
There exists a constant $c>0$ such that for all $\al\in (0,2)$ and $\ka\in [0,1)$, one has, as $n$ tends to infinity,
\begin{equation}\label{asymptotic-exp}
\int_0^{2^n} \frac{dr}{r^{-\al-\ka}} \int_{|\xi|\leq 2^n} \frac{d\xi}{|\xi|^{\al-1}} |\ga_t(\xi,r)|^2 =c\,  t \int_1^{2^n} \frac{dr}{r^{1-\ka}} +O(1) \ .
\end{equation}
\end{proposition}

\begin{proof}
First, observe that using (\ref{estimate-gamma}), we have
\begin{align*}
&\bigg| \int_0^{1} \frac{dr}{r^{-\al-\ka}} \int_{|\xi|\leq 2^n} \frac{d\xi}{|\xi|^{\al-1}} |\ga_t(\xi,r)|^2 \bigg|\\
&  \lesssim  \{1+t^4\} \int_0^1 \frac{dr}{r^{-2H_0-\ka}} \int_{|\xi|\leq 1} \frac{d\xi}{|\xi|^{\al-1}}+t^2 \int_0^1 \frac{dr}{r^{-\al-\ka}} \int_{|\xi|\geq 1} \frac{d\xi}{|\xi|^{1+\al}} \ ,
\end{align*}
and accordingly it suffices to focus on the estimation of the integral
$$\int_1^{2^n} \frac{dr}{r^{-\al-\ka}} \int_{|\xi|\leq 2^n} \frac{d\xi}{|\xi|^{\al-1}} |\ga_t(\xi,r)|^2 \ .$$
To this end, we will rely on the following expansion, which can be readily derived from (\ref{exact-expression-gamma}):
\begin{eqnarray*}
|\ga_t(\xi,r)|^2&=&\frac{c}{r^2} \bigg\{ \frac{1-\cos(t(\xi-r))}{(\xi-r)^2} -\frac{\cos(tr)\{\cos(tr)-\cos(t\xi)\}}{(\xi-r)(\xi+r)} \bigg\}\\
& &+\frac{c}{r^2} \bigg\{ \frac{1-\cos(t(\xi+ r))}{(\xi+ r)^2} -\frac{\cos(tr)\{\cos(tr)-\cos(t\xi)\}}{(\xi-r)(\xi+r)} \bigg\}\\
&=:& \Gamma_t(\xi,r)+\widetilde{\Gamma}_t(\xi,r) \ .
\end{eqnarray*}
For obvious symmetry reasons, we have in fact
\begin{eqnarray}
\lefteqn{\int_1^{2^n} \frac{dr}{r^{-\al-\ka}} \int_{|\xi|\leq 2^n} \frac{d\xi}{|\xi|^{\al-1}} |\ga_t(\xi,r)|^2}\nonumber\\
&= & 2\int_1^{2^n} \frac{dr}{r^{-\al-\ka}} \int_{|\xi|\leq 2^n} \frac{d\xi}{|\xi|^{\al-1}} \Gamma_t(\xi,r)\nonumber \\
&=&2\int_1^{2^n} \frac{dr}{r^{-\al-\ka}} \int_0^r \frac{d\xi}{|\xi|^{\al-1}} \Gamma_t(\xi,r)+2\int_1^{2^n} \frac{dr}{r^{-\al-\ka}} \int_r^{2^n} \frac{d\xi}{|\xi|^{\al-1}} \Gamma_t(\xi,r)\nonumber\\
& &\hspace{1cm}+2\int_1^{2^n} \frac{dr}{r^{-\al-\ka}} \int_0^{2^n} \frac{d\xi}{|\xi|^{\al-1}} \Gamma_t(-\xi,r) \ =: \ \mathcal{J}^1_{n,t} +\mathcal{J}^2_{n,t}+\mathcal{J}^3_{n,t} \ . \label{decompo-mathcal-j}
\end{eqnarray}

\

\noindent
\textbf{Study of $\mathcal{J}^1_{n,t}$.} Let us introduce the additional notation
$$\Gamma^1_t(\xi,r):=  \frac{1-\cos(t(\xi-r))}{r^2(\xi-r)^2}  \quad , \quad \Gamma^2_t(\xi,r):=\frac{\cos(tr)\{\cos(t\xi)-\cos(tr)\}}{r^2(\xi-r)(\xi+r)} \ ,$$
so that
\begin{equation}\label{decompo-gamma-1}
\Gamma_t(\xi,r)=c\big\{\Gamma^1_t(\xi,r)+\Gamma^2_t(\xi,r)\big\} \ .
\end{equation}
Now on the one hand, for any $0<\varepsilon <1$,
\begin{eqnarray*}
\lefteqn{\bigg| \int_1^{2^n} \frac{dr}{r^{-\al-\ka}} \int_0^r \frac{d\xi}{|\xi|^{\al-1}} \Gamma^2_t(\xi,r)\bigg|}\\
&=& \bigg| \int_1^{2^n} \frac{dr}{r^{2-\ka}} \int_0^1 \frac{d\xi}{|\xi|^{\al-1}} \frac{\cos(tr)(\cos(tr\xi)-\cos(tr))}{(\xi-1)(\xi+1)} \bigg|\\
&\lesssim& t^\varepsilon \int_1^\infty \frac{dr}{r^{2-\ka-\varepsilon}} \int_0^1 \frac{d\xi}{|\xi|^{\al-1} |1-\xi|^{1-\varepsilon} |1+\xi|}
\end{eqnarray*}
and the latter integrals are finite for any $\varepsilon >0$ such that $\ka+\varepsilon <1$, which shows that
$$\int_1^{2^n} \frac{dr}{r^{-\al-\ka}} \int_0^r \frac{d\xi}{|\xi|^{\al-1}} \Gamma^2_t(\xi,r)=O(1) \ .$$
On the other hand,
\begin{eqnarray*}
\int_1^{2^n} \frac{dr}{r^{-\al-\ka}} \int_0^r \frac{d\xi}{|\xi|^{\al-1}} \Gamma^1_t(\xi,r)
&=& \int_1^{2^n} \frac{dr}{r^{2-\ka}} \int_0^1 \frac{d\xi}{|\xi|^{\al -1}} \frac{1-\cos(tr(1-\xi))}{(1-\xi)^2} \\
&=& \int_1^{2^n} \frac{dr}{r^{2-\ka}} \int_0^{\frac12} \frac{d\xi}{|1-\xi|^{\al -1}} \frac{1-\cos(tr\xi)}{\xi^2}\\
& &\hspace{2cm}+\int_1^{2^n} \frac{dr}{r^{2-\ka}} \int_{\frac12}^1 \frac{d\xi}{|1-\xi|^{\al -1}} \frac{1-\cos(tr\xi)}{\xi^2} \ ,
\end{eqnarray*}
with
$$\bigg| \int_1^{2^n} \frac{dr}{r^{2-\ka}} \int_{\frac12}^1 \frac{d\xi}{|1-\xi|^{\al -1}} \frac{1-\cos(tr\xi)}{\xi^2} \bigg| \lesssim \int_1^\infty \frac{dr}{r^{2-\ka}} \int_{\frac12}^1 \frac{d\xi}{|1-\xi|^{\al -1}}\, < \, \infty $$
and
\begin{eqnarray*}
\lefteqn{\int_1^{2^n} \frac{dr}{r^{2-\ka}} \int_0^{\frac12} \frac{d\xi}{|1-\xi|^{\al -1}} \frac{1-\cos(tr\xi)}{\xi^2}}\\
 &=& t \int_1^{2^n} \frac{dr}{r^{1-\ka}} \int_0^{\frac{rt}{2}} \frac{d\xi}{|1-\frac{\xi}{rt}|^{\al -1}} \frac{1-\cos(\xi)}{\xi^2}\\
&=&t \int_0^{\infty } d\xi  \frac{1-\cos(\xi)}{\xi^2}\int_1^{2^n} \frac{dr}{r^{1-\ka}}\\
& &+t \int_1^{2^n} \frac{dr}{r^{1-\ka}}\bigg[  \int_0^{\frac{rt}{2}} \frac{d\xi}{|1-\frac{\xi}{rt}|^{\al -1}} \frac{1-\cos(\xi)}{\xi^2}-\int_0^{\infty } d\xi  \frac{1-\cos(\xi)}{\xi^2}\bigg] \ .
\end{eqnarray*}
By applying the below technical Lemma \ref{lem:tech}, we can easily conclude that
\begin{equation}\label{expan-j-1}
\mathcal{J}^1_{n,t}=c \, t\int_1^{2^n} \frac{dr}{r^{1-\ka}}+O(1) \ .
\end{equation}

\

\noindent
\textbf{Study of $\mathcal{J}^2_{n,t}$.} We will here use the (readily-checked) decomposition
\begin{equation}\label{decompo-gamma-2}
\Gamma_t(\xi,r)=c\big\{2\Gamma^3_t(\xi,r)+\Gamma^4_t(\xi,r)\big\} 
\end{equation}
with
$$\Gamma^3_t(\xi,r):=\frac{1-\cos(t(\xi-r))}{r(\xi-r)^2(\xi+r)}$$
and
$$\Gamma^4_t(\xi,r):=\frac{1-\cos(t(\xi-r))-\cos(tr)(\cos(tr)-\cos(t\xi))}{r^2(\xi-r)(\xi+r)} \ .$$
Now observe on the one hand that for every $\varepsilon \in (0,1)$,
\begin{eqnarray*}
\bigg| \int_1^{2^n} \frac{dr}{r^{-\al-\ka}} \int_r^{2^n} \frac{d\xi}{|\xi|^{\al-1}} \Gamma^4_t(\xi,r) \bigg|
&\lesssim &t^\varepsilon \int_1^{2^n} \frac{dr}{r^{2-\al-\ka}} \int_r^\infty \frac{d\xi}{|\xi|^{\al-1} |\xi-r|^{1-\varepsilon} |\xi+r|}\\
& \lesssim &  t^\varepsilon \int_1^{\infty} \frac{dr}{r^{2-\ka-\varepsilon}} \int_1^\infty \frac{d\xi}{|\xi|^{\al-1} |1-\xi|^{1-\varepsilon} |1+\xi|}
\end{eqnarray*}
and the latter integrals are finite provided $0<\varepsilon <\min(1-\ka,\al)$. On the other hand,
\begin{eqnarray}
\lefteqn{\int_1^{2^n} \frac{dr}{r^{-\al-\ka}} \int_r^{2^n} \frac{d\xi}{|\xi|^{\al-1}} \Gamma^3_t(\xi,r)}\nonumber\\
& = & \int_1^{2^n} \frac{dr}{r^{1-\al-\ka}} \int_0^{2^n-r} \frac{d\xi}{|\xi+r|^{\al-1}|\xi+2r|} \frac{1-\cos(t\xi)}{\xi^2}\nonumber\\
&=& \frac{t}{2}\int_1^{2^n} \frac{dr}{r^{1-\ka}} \int_0^{\infty}d\xi \frac{1-\cos(\xi)}{\xi^2}+\int_1^{2^n} \frac{dr}{r^{1-\al-\ka}} \nonumber\\
& &\hspace{0.5cm}\bigg[\int_0^{2^n-r} \frac{d\xi}{|\xi+r|^{\al-1}|\xi+2r|} \frac{1-\cos(t\xi)}{\xi^2}-\frac{1}{2r^\al}\int_0^{\infty}d\xi \frac{1-\cos(t\xi)}{\xi^2} \bigg] \ .\label{interm-i-n-t-2}
\end{eqnarray}
Using Lemma \ref{lem:tech-2}, we can then assert that for $\varepsilon \in (0,\min(\al,1-\ka))$, 
\small
\begin{align*}
&\bigg| \int_1^{2^n} \frac{dr}{r^{1-\al-\ka}} \bigg[\int_0^{2^n-r} \frac{d\xi}{|\xi+r|^{\al-1}|\xi+2r|} \frac{1-\cos(t\xi)}{\xi^2}-\frac{1}{2r^\al}\int_0^{\infty}d\xi \frac{1-\cos(t\xi)}{\xi^2} \bigg] \bigg|\\
&\lesssim t^\varepsilon\int_1^{2^n} \frac{dr}{r^{2-\ka-\varepsilon}} +t^\varepsilon\int_1^{2^n} \frac{dr}{r^{1-\ka}|2^n-r|^{1-\varepsilon}}\\
& \lesssim  t^\varepsilon\int_1^{\infty} \frac{dr}{r^{2-\ka-\varepsilon}} +t^\varepsilon 2^{-n(1-\ka-\varepsilon)}\int_0^1 \frac{dr}{r^{1-\ka}|1-r|^{1-\varepsilon}} \ .
\end{align*}
\normalsize
Going back to (\ref{interm-i-n-t-2}), we have thus shown that
\begin{equation}\label{expan-j-2}
\mathcal{J}^2_{n,t}=c \, t\int_1^{2^n} \frac{dr}{r^{1-\ka}}+O(1) \ .
\end{equation}

\

\noindent
\textbf{Study of $\mathcal{J}^3_{n,t}$.} Using decomposition (\ref{decompo-gamma-1}), it is readily checked that for $\varepsilon \in (0,1)$ and for all $\xi,r>0$,
$$\big| \Gamma_t(-\xi,r)\big| \lesssim \frac{1}{r^2} \bigg[ \frac{1}{|\xi+r|^2} +\frac{t^\varepsilon}{|\xi+r| |\xi-r|^{1-\varepsilon}} \bigg] $$
and so
\begin{eqnarray*}
\bigg|\int_1^{2^n} \frac{dr}{r^{-\al-\ka}} \int_0^{2^n} \frac{d\xi}{|\xi|^{\al-1}} \Gamma_t(-\xi,r)\bigg| &\lesssim & \int_1^{2^n} \frac{dr}{r^{2-\al-\ka}} \int_0^{2^n} \frac{d\xi}{|\xi|^{\al-1}}  \bigg[\frac{1}{|\xi+r|^2} + \frac{t^\varepsilon}{|\xi+r| |\xi-r|^{1-\varepsilon}} \bigg]\\
&\lesssim & \int_1^{\infty} \frac{dr}{r^{2-\ka-\varepsilon}} \int_0^{\infty} \frac{d\xi}{|\xi|^{\al-1}}  \bigg[\frac{1}{|\xi+1|^2} + \frac{t^\varepsilon}{|\xi+1| |\xi-1|^{1-\varepsilon}} \bigg] \ .
\end{eqnarray*}
The latter integrals being finite as soon as $0<\varepsilon <\min(1-\ka,\al)$, this shows that $\mathcal{J}^3_{n,t}=O(1)$. Injecting this result, together with (\ref{expan-j-1}) and (\ref{expan-j-2}), into (\ref{decompo-mathcal-j}) yields the expected decomposition (\ref{asymptotic-exp}).
\end{proof}

\smallskip

\begin{lemma}\label{lem:tech}
Given $\al\in (0,2)$ and $\varepsilon \in (0,1)$, one has, for all $r>0$, 
$$\bigg| \int_0^{\frac{r}{2}} \frac{d\xi}{|1-\frac{\xi}{r}|^{\al -1}} \frac{1-\cos(\xi)}{\xi^2}-\int_0^{\infty } d\xi  \frac{1-\cos(\xi)}{\xi^2} \bigg| \leq c_{\al,\varepsilon} \bigg[\frac{1}{r}+\frac{1}{r^{1-\varepsilon}}\bigg] \ .$$
\end{lemma}

\begin{proof}
Write the difference as
$$\int_0^{\frac{r}{2}}d\xi \bigg[ \frac{1}{|1-\frac{\xi}{r}|^{\al -1}}-1\bigg] \frac{1-\cos(\xi)}{\xi^2}+\int_{\frac{r}{2}}^{\infty } d\xi  \frac{1-\cos(\xi)}{\xi^2} \ .$$
Then observe that
$$\bigg| \int_{\frac{r}{2}}^{\infty } d\xi  \frac{1-\cos(\xi)}{\xi^2} \bigg| \lesssim \frac{1}{r^{1-\varepsilon}} \int_{0}^{\infty } d\xi  \frac{1-\cos(\xi)}{|\xi|^{1+\varepsilon}} $$
and
\begin{eqnarray*}
\bigg| \int_0^{\frac{r}{2}}d\xi \bigg[ \frac{1}{|1-\frac{\xi}{r}|^{\al -1}}-1\bigg] \frac{1-\cos(\xi)}{\xi^2} \bigg| &\lesssim & \int_0^{\frac{r}{2}}d\xi \frac{| |r|^{\al-1}-|r-\xi|^{\al-1}|}{|r-\xi|^{\al -1}}  \frac{1-\cos(\xi)}{\xi^2}\\
&\lesssim & \int_0^{\frac{r}{2}}d\xi   \frac{1-\cos(\xi)}{|r-\xi| |\xi|} \ \lesssim \ \frac{1+\lln \log r\rrn 1_{\{r>1\}}}{r} \ ,
\end{eqnarray*}
hence the conclusion.
\end{proof}

\smallskip

\begin{lemma}\label{lem:tech-2}
Given $\al\in (0,2)$ and $0<\varepsilon <\min(\al,1)$, one has, for all $t>0$, $n\geq 0$ and $1< r <2^n$,
$$
\bigg|\int_0^{2^n-r} \frac{d\xi}{|\xi+r|^{\al-1}|\xi+2r|} \frac{1-\cos(t\xi)}{\xi^2}-\frac{1}{2r^\al}\int_0^{\infty}d\xi \frac{1-\cos(t\xi)}{\xi^2}  \bigg| \leq c_{\al,\varepsilon}\bigg[\frac{t^\varepsilon}{r^{1+\al-\varepsilon}}+\frac{t^\varepsilon}{r^\al |2^n-r|^{1-\varepsilon}}\bigg] \ . 
$$
\end{lemma}

\begin{proof}
Let us decompose the difference as
\begin{align*}
&\int_0^{2^n-r} d\xi \bigg[ \frac{1}{|\xi+r|^{\al-1}}- \frac{1}{r^{\al-1}} \bigg] \frac{1-\cos(t\xi)}{|\xi+2r|\xi^2}\\
&\hspace{3cm}+\frac{1}{r^{\al-1}}\int_0^{2^n-r}d\xi \bigg[ \frac{1}{|\xi+2r|}- \frac{1}{2r} \bigg] \frac{1-\cos(t\xi)}{\xi^2} -\frac{1}{2r^\al} \int_{2^n-r}^\infty d\xi \, \frac{1-\cos(t\xi)}{\xi^2} \ .
\end{align*}
Then observe that
\begin{eqnarray*}
\bigg| \int_0^{2^n-r} d\xi \bigg[ \frac{1}{|\xi+r|^{\al-1}}- \frac{1}{r^{\al-1}} \bigg] \frac{1-\cos(t\xi)}{|\xi+2r|\xi^2} \bigg|
&\lesssim &\change{ \frac{1}{r^\al} \int_0^\infty \frac{d\xi}{|\xi+2r|} \frac{1-\cos(t\xi)}{|\xi|}} \\
&\lesssim &\change{\frac{1}{r^{1+\al-\varepsilon}} \int_0^\infty d\xi \frac{1-\cos(t\xi)}{|\xi|^{1+\varepsilon}} \, ,}
\end{eqnarray*}
and
\begin{eqnarray*}
\frac{1}{r^{\al-1}}\bigg| \int_0^{2^n-r}d\xi \bigg[ \frac{1}{|\xi+2r|}- \frac{1}{2r} \bigg] \frac{1-\cos(t\xi)}{\xi^2} \bigg|&=&\frac{1}{2r^{\al}}\bigg| \int_0^{2^n-r}d\xi  \frac{\xi}{|\xi+2r|} \frac{1-\cos(t\xi)}{\xi^2} \bigg|\\
&\lesssim & \frac{1}{r^{1+\al-\varepsilon}} \int_0^\infty d\xi \frac{1-\cos(t\xi)}{|\xi|^{1+\varepsilon}} \ .
\end{eqnarray*}
Finally, one has of course
$$\bigg| \int_{2^n-r}^{\infty } d\xi  \frac{1-\cos(t\xi)}{\xi^2} \bigg| \lesssim \frac{1}{|2^n-r|^{1-\varepsilon}} \int_{0}^{\infty } d\xi  \frac{1-\cos(t\xi)}{|\xi|^{1+\varepsilon}} \ .$$
\end{proof}

\section{Study of the (deterministic) auxiliary equation}\label{sec:auxiliary-equation}

Let us now turn to the analysis of the deterministic equation associated with our quadratic model (\ref{1-d-quadratic-wave}), that is the equation
\change{\begin{equation}\label{eq-deter-base-bis}
\left\{
\begin{array}{l}
\partial^2_t v -\Delta v+ \rho^2 v^2+(\rho v)\cdot \mathbf{\Pi}^{\mathbf{1}}+\mathbf{\Pi}^{\mathbf{2}}=0\, , \quad\quad  t\in [0,T] \, , \ x\in \R^d \, ,\\
v(0,.)=\phi_0 \ , \ \partial_t v(0,.)=\phi_1 \, , 
\end{array}
\right.
\end{equation}}
where $\mathbf{\Pi}^{\mathbf{1}},\mathbf{\Pi}^{\mathbf{2}}$ are two (fixed) elements living in appropriate Sobolev spaces. We are actually interested in the exhibition of a unique (local) mild solution to (\ref{eq-deter-base-bis}), which will be achieved by means of a standard fixed-point argument. In other words, for fixed $\mathbf{\Pi}\triangleq (\mathbf{\Pi}^{\mathbf{1}},\mathbf{\Pi}^{\mathbf{2}})$ and $T>0$, we will focus on the study of the map $\Gamma_{T,\mathbf{\Pi}}$ defined as
\begin{equation}\label{defi-gamma-phi-t}
\boxed{\Gamma_{T,\mathbf{\Pi}}(v)_t\triangleq \partial_t(G_t \ast_x \phi_0)+G_t\ast_x \phi_1 +(G \ast_{t,x} [\rho^2 v^2+ \mathbf{\Pi}^{\mathbf{1}} \cdot \rho v+\mathbf{\Pi}^{\mathbf{2}}])_t} \ ,
\end{equation}
where $G$ stands for the Green function of the standard $d$-dimensional wave equation. Putting the fixed components aside, this map is thus essentially built upon two successive operations: multiplication of $v$ with itself or with $\mathbf{\Pi}^{\mathbf{1}}$, and convolution with $G$. Accordingly, before we specify the space in which we will study $\Gamma_{T,\mathbf{\Pi}}$, let us recall a few general results on pointwise multiplication and convolution with the wave kernel.

\subsection{Pointwise multiplication}
Recall that, with the results of Section \ref{sec:stochastic} in mind, one of our purposes is to handle situations where the elements $\mathbf{\Pi}^{\mathbf{1}},\mathbf{\Pi}^{\mathbf{2}}$ involved in (\ref{defi-gamma-phi-t}) are not functions but only distributions. Thus, even if  we expect the solution $v$ itself to be a function, we will need to control the (non-standard) multiplication of a function with a distribution at some point of the procedure. To this end, we can rely on the following general statement, borrowed from \cite[Section 4.5.1]{runst-sickel}.


\begin{proposition}\label{product}
Fix $d\geq 1$. Let $\al,s >0$ and {$1 \leq p,p_1,p_2< \infty$} be such that 
$$
\frac{1}{p} = \frac{1}{p_{1}}+\frac{1}{p_{2}}\quad \text{and} \quad 0<\alpha<s.
$$
If $f\in \cw^{-\al,p_1}(\R^d)$ and $g\in \cw^{s,p_2}(\R^d)$, then $f\cdot g\in \cw^{-\al,p}(\R^d)$ and
$$
\| f\cdot g\|_{\cw^{-\al,p}} \lesssim \|f\|_{\cw^{-\al,p_1}} \| g\|_{\cw^{s,p_2}} \ .$$

\end{proposition}

\

\subsection{Regularization effect of the wave operator}

The wave kernel $G$ is known to satisfy specific regularization properties in the scale of Sobolev spaces (see e.g. \cite{ginibre-velo}). In the analysis of \eqref{eq-deter-base-bis}, we will only rely on the following controls.

\begin{proposition}\label{prop:regu-wave}
Fix $d\geq 2$.

\smallskip

\noindent
$(i)$ For all $0\leq T \leq 1$, $s>0$ and $w\in L^2([0,T];\ch^{s-1}(\R^d))$, it holds that
\begin{equation}\label{basic-strichartz-1}
\cn\big[ G \ast_{t,x} w;L^\infty([0,T];\ch^s(\R^d))\big] \lesssim \cn[w;L^2([0,T];\ch^{s-1}(\R^d))] \ .
\end{equation}
$(ii)$ For all $0\leq T \leq 1$, $s>0$ and $(\phi_0,\phi_1) \in \ch^s(\R^d) \times \ch^{s-1}(\R^d)$, it holds that
\begin{equation}\label{basic-strichartz-2}
\cn[ \partial_t(G_t \ast_x \phi_0);L^\infty([0,T];\ch^s(\R^d))]\lesssim \| \phi_0\|_{\ch^{s}(\R^d)}
\end{equation}
and
\begin{equation}
\cn[ G_t\ast_x \phi_1 ;L^\infty([0,T];\ch^s(\R^d))]\lesssim \| \phi_1\|_{\ch^{s-1}(\R^d)} \, .
\end{equation}
\end{proposition}

\begin{proof}
Both \eqref{basic-strichartz-1} and \eqref{basic-strichartz-2} lean on elementary estimates.

\smallskip

\noindent
$(i)$ Observe first that
$$\cf_x \big( \big[ G\ast_{t,x} w\big](t,.)\big)(\xi)=\int_0^t ds \, \big( \cf_x G_{t-s}\big)(\xi) \big( \cf_x w_s\big) (\xi) \ ,$$
so that, for every $t\in [0,T]$,
\begin{eqnarray*}
\int_{\R^d} d\xi \, \{1+|\xi|^2\}^{s} \big|\cf_x\big( \big[ G\ast_{t,x} w\big](t,.)\big)(\xi)\big|^2 & \leq & \int_0^t ds \int_{\R^2} d\xi \, \{1+|\xi|^2\}^{s} \frac{\sin^2((t-s)|\xi|)}{|\xi|^2} | ( \cf_x w_s) (\xi)|^2\\
&\lesssim& \int_0^t ds \int_{\R^d} d\xi \, \{1+|\xi|^2\}^{s-1}  | (\cf_x w_s)(\xi)|^2  \ ,
\end{eqnarray*}
hence the conclusion.

\smallskip

\noindent
$(ii)$ The bound for $\cn[ G_t\ast_x \phi_1 ;L^\infty([0,T];\ch^{s}(\R^d))]$ follows from similar arguments as above. Then, along the same idea, observe that
$$\cf_x\big( \partial_t(G_. \ast_x \phi_0)(t,.)\big)(\xi)=\partial_t ( \cf_x G)_t(\xi) (\cf_x \phi_0)(\xi)=\cos(t|\xi|) (\cf_x \phi_0)(\xi) \ ,$$
and so, for every $t\in [0,T]$,
$$
\int_{\R^2} d\xi \, \{1+|\xi|^2\}^{s} \big|\cf_x\big( \partial_t(G_. \ast_x \phi_0)(t,.)\big)(\xi)\big|^2  \leq  \int_{\R^2} d\xi \, \{1+|\xi|^2\}^{s} |(\cf_x \phi_0)(\xi)|^2 \ .
$$

\end{proof}

\subsection{First situation}
Let us first consider the situation where the pair $\mathbf{\Pi}=(\mathbf{\Pi}^{\mathbf{1}},\mathbf{\Pi}^{\mathbf{2}})$ involved in (\ref{eq-deter-base-bis}) (or in (\ref{defi-gamma-phi-t})) belongs to the space
\begin{align*}
&\mathcal{E}\triangleq \big\{\mathbf{\Pi}=(\mathbf{\Pi}^{\mathbf{1}},\mathbf{\Pi}^{\mathbf{2}}) \in L^{\infty}([0,T];L^\infty(\change{\R^d}))^2:\\
&\hspace{2cm}  \text{$\mathbf{\Pi}^{\mathbf{1}}$ and $\mathbf{\Pi}^{\mathbf{2}}$ are both supported by a compact set $D$}\big\} \ .
\end{align*}
Thus, for the moment, $\mathbf{\Pi}^{\mathbf{1}}$ and $\mathbf{\Pi}^{\mathbf{2}}$ are merely (bounded and compactly-supported) \textit{functions}. When going back to the stochastic model (\ref{1-d-quadratic-wave}) and with the result of Proposition \ref{prop:regu-psi} in mind, this situation will later correspond to the \enquote{regular} case $\sum_{i=0}^d H_i>d-\frac12$ (along the splitting of Theorem \ref{main-theo} or Theorem \ref{main-theo-lim}).

\begin{proposition}\label{prop:deterministic-result-d-i}
Fix $d\in \{2,3\}$ and set
\begin{equation*}
X(T)\triangleq L^\infty([0,T];\ch^1(\R^d)) \, .
\end{equation*}
Then, for all $T>0$, $(\phi_0,\phi_1) \in \ch^1(\R^d) \times L^2(\R^d)$, $\mathbf{\Pi}_1=(\mathbf{\Pi}_1^{\mathbf{1}},\mathbf{\Pi}_1^{\mathbf{2}})\in \ce$, $\mathbf{\Pi}_2=(\mathbf{\Pi}_2^{\mathbf{1}},\mathbf{\Pi}_2^{\mathbf{2}})\in \mathcal{E}$ and $v,v_1,v_2 \in X(T)$, the following bounds hold true:
\begin{equation}\label{boun-gamma-1}
\cn[\Gamma_{T,\mathbf{\Pi}_1}(v);X(T)] \lesssim \| \phi_0\|_{\ch^1}+\| \phi_1\|_{L^2}+T^{\frac12} \cn[v;X(T)]^2+ T^{\frac12}  \|\mathbf{\Pi}_1\|\, \cn[v;X(T)]+T^{\frac12} \|\mathbf{\Pi}_1\| \, ,
\end{equation}
and
\begin{align}
&\cn[\Gamma_{T,\mathbf{\Pi}_1}(v_1)-\Gamma_{T,\mathbf{\Pi}_2}(v_2);X(T)]\nonumber\\
&\hspace{1cm}\lesssim T^{\frac12} \cn[v_1-v_2;X(T)]\{\cn[v_1;X(T)]+\cn[v_2;X(T)]\}\nonumber\\
 &\hspace{2cm}+ T^{\frac12}  \|\mathbf{\Pi}_1-\mathbf{\Pi}_2\|\, \cn[v_1;X(T)]+ T^{\frac12}  \|\mathbf{\Pi}_2\|\, \cn[v_1-v_2;X(T)]+T^{\frac12} \|\mathbf{\Pi}_1-\mathbf{\Pi}_2\| \, , \label{boun-gamma-2}
\end{align}
where the proportional constants only depend on $s$ and the norm $\|.\|$ is naturally defined as
$$\|\mathbf{\Pi}\|=\|\mathbf{\Pi}\|_{\mathcal{E}}\triangleq\cn[\mathbf{\Pi}^{\mathbf{1}} ;L^{\infty}([0,T];L^\infty(\change{\R^d}))]+\cn[\mathbf{\Pi}^{\mathbf{2}} ;L^{\infty}([0,T];L^\infty(\change{\R^d}))] \ .$$
\end{proposition}

\

By combining the two bounds \eqref{boun-gamma-1} and \eqref{boun-gamma-2}, it is now easy to see that for any fixed $\mathbf{\Pi} \in \mathcal{E}$ and any time $T_0>0$ small enough, the map $\Gamma_{T_0,\mathbf{\Pi}}: X(T_0) \to X(T_0)$ is a contraction on a appropriate stable ball of $X(T_0)$, which immediately yields the expected (local) well-posedness result: 
\begin{corollary}\label{cor:deterministic-result-d-i}
Under the assumptions of Proposition \ref{prop:deterministic-result-d-i}, and for all (fixed) $(\phi_0,\phi_1) \in \ch^1(\R^d) \times L^2(\R^d)$, $\mathbf{\Pi}=(\mathbf{\Pi}^{\mathbf{1}},\mathbf{\Pi}^{\mathbf{2}})\in \mathcal{E}$, there exists a time $T_0>0$ such that Equation \eqref{eq-deter-base-bis} admits a unique solution in $X(T_0)$.
\end{corollary}

\smallskip

\begin{proof}[Proof of Proposition \ref{prop:deterministic-result-d-i}]

Using the estimates in Proposition \ref{prop:regu-wave}, we obtain first
\begin{align}
&\cn[\Gamma_{T,\mathbf{\Pi}_1}(v);X(T)] \lesssim \| \phi_0\|_{\ch^1}+\| \phi_1\|_{L^2}+\cn[\rho^2 v^2;L^{2}([0,T];L^2(\R^d))]\nonumber\\
&\hspace{4cm}+\cn[\mathbf{\Pi}^{\mathbf{1}}_1 \cdot \rho v;L^{2}([0,T];L^{2}(\R^d))]+\cn[\mathbf{\Pi}^{\mathbf{2}}_1 ;L^{2}([0,T];L^{2}(\R^d))]\, .\label{appli-wave-op}
\end{align}
In order to bound the last three quantities, we will merely appeal to the elementary continuous embedding: for all compact domain $D\subset \R^d$, $p_0,p_1\geq 1$ and $s_0\geq s_1$,
\begin{equation}\label{sobol-emb}
\cw^{s_0,p_0}(D) \subset \cw^{s_1,p_1}(D) \quad \text{if} \ s_0-\frac{d}{p_0} \geq s_1-\frac{d}{p_1} \, .
\end{equation}

\

\noindent
\textbf{Bound on $\cn[\rho^2 v^2;L^{2}([0,T];L^2(\R^d))]$.} As $\rho$ is supported by some fixed compact domain $D$, one has
$$\cn[\rho^2 v^2;L^{2}([0,T];L^2(\R^d))] =\cn[\rho^2 v^2;L^{2}([0,T];L^2(D))] \lesssim T^{\frac12}\cn[v;L^{\infty}([0,T];L^4(D))]^2 \, .$$
Since $d\in \{2,3\}$, we can assert, by \eqref{sobol-emb}, that $\ch^1(D)\subset L^4(D)$, and thus
$$\cn[ v;L^{\infty}([0,T];L^4(D))]\lesssim \cn[v;X(T)] \, .$$

\

\noindent
\textbf{Bound on $\cn[\mathbf{\Pi}^{\mathbf{1}}_1 \cdot \rho v;L^{2}([0,T];L^{2}(\R^d))]$.} One has of course
$$\cn[\mathbf{\Pi}^{\mathbf{1}}_1 \cdot \rho v;L^{2}([0,T];L^{2}(\R^d))]\lesssim  T^{\frac12}\cn[\mathbf{\Pi}^{\mathbf{1}}_1 ;L^{\infty}([0,T];L^\infty(\R^d))] \, \cn[\rho v;L^{\infty}([0,T];L^2(\R^d))]\, ,$$
and then
$$ \cn[\rho v;L^{\infty}([0,T];L^2(\R^d))]\lesssim \cn[v;L^{\infty}([0,T];L^2(\R^d))]\lesssim  \cn[v;X(T)]\, .$$

\

\noindent
\textbf{Bound on $\cn[\mathbf{\Pi}^{\mathbf{2}}_1 ;L^{2}([0,T];L^{2}(\R^d))]$.} Since $\mathbf{\Pi}^{\mathbf{2}}_1$ is compactly-supported, we immediately have
$$
\cn[\mathbf{\Pi}^{\mathbf{2}}_1 ;L^{2}([0,T];L^{2}(\R^d))] \lesssim T^{\frac12} \cn[\mathbf{\Pi}^{\mathbf{2}}_1 ;L^{\infty}([0,T];L^{\infty}(\R^d))]\, .
$$

\

Injecting the above estimates into \eqref{appli-wave-op} provides us with \eqref{boun-gamma-1}. It is then clear that \eqref{boun-gamma-2} can be derived from similar arguments: for instance,
\begin{align*}
\cn[\rho^2 (v_1^2-v_2^2);L^{2}([0,T];L^2(\R^d))] &\lesssim  T^{\frac12}\cn[v_1^2-v_2^2;L^{\infty}([0,T];L^2(D))] \\
&\lesssim  T^{\frac12}\cn[v_1-v_2;L^{\infty}([0,T];L^{4}(D))] \cn[v_1+v_2;L^{\infty}([0,T];L^{4}(D))]\\
&\lesssim  T^{\frac12} \cn[v_1-v_2;X(T)] \cn[v_1+v_2;X(T)] \ .
\end{align*}

\end{proof}

\subsection{Second situation}
We now turn to the \enquote{irregular} case of our analysis, that will later correspond to item $(ii)$ in Theorem \ref{main-theo} or Theorem \ref{main-theo-lim}. With the result of Proposition \ref{prop:regu-psi-order-two} in mind, we are thus led to consider the situation where the pair $\mathbf{\Pi}=(\mathbf{\Pi}^{\mathbf{1}},\mathbf{\Pi}^{\mathbf{2}})$ in (\ref{defi-gamma-phi-t}) belongs to the space 
\begin{align*}
&\mathcal{E}_{\al,p}\triangleq \big\{\mathbf{\Pi}=(\mathbf{\Pi}^{\mathbf{1}},\mathbf{\Pi}^{\mathbf{2}}) \in L^{\infty}([0,T];\cw^{-\al,p}(\change{\R^d})) \times L^{\infty}([0,T];\cw^{-2\al,p}(\change{\R^d})):\\
&\hspace{2cm}  \text{$\mathbf{\Pi}^{\mathbf{1}}$ and $\mathbf{\Pi}^{\mathbf{2}}$ are both supported by a compact set $D$}\big\} \ ,
\end{align*}
for some positive coefficient $\al$ and some integer $p\geq 2$. In particular, $\mathbf{\Pi}^{\mathbf{1}}$ and $\mathbf{\Pi}^{\mathbf{2}}$ are now both regarded as distributions. Our main result in this setting reads as follows.

\begin{proposition}\label{prop:deterministic-result-d}
Fix $d\in \{2,3\}$, $\al\in (0,\frac14)$ and $s\in (0,1)$ such that
\begin{equation}\label{encadr-s}
s\geq \frac{d}{2}-1 \quad \text{and}\quad \al< s < 1-2\al\, .
\end{equation}
Set
\begin{equation*}\label{defi-x-s-t}
X^s(T)\triangleq L^\infty([0,T];\ch^s(\R^d)) \, ,
\end{equation*} 
and let $p>d$ be defined by the relation
\begin{equation}\label{defi-p-rough}
\frac{1}{p}=\frac{1-(\al+s)}{d} \, .
\end{equation}
Then, for all $T>0$, $(\phi_0,\phi_1) \in \ch^{s}(\R^d) \times \ch^{s-1}(\R^d)$, $\mathbf{\Pi}_1=(\mathbf{\Pi}_1^{\mathbf{1}},\mathbf{\Pi}_1^{\mathbf{2}})\in \ce_{\al,p}$, $\mathbf{\Pi}_2=(\mathbf{\Pi}_2^{\mathbf{1}},\mathbf{\Pi}_2^{\mathbf{2}})\in \mathcal{E}_{\al,p}$ and $v,v_1,v_2 \in X^s(T)$,
the following bounds hold true:
\begin{equation*}
\cn[\Gamma_{T,\mathbf{\Pi}_1}(v);X^s(T)] \lesssim \| \phi_0\|_{\ch^s}+\| \phi_1\|_{\ch^{s-1}}+T^{\frac12} \cn[v;X^s(T)]^2+ T^{\frac12}  \|\mathbf{\Pi}_1\|\, \cn[v;X^s(T)]+T^{\frac12} \|\mathbf{\Pi}_1\| \, ,
\end{equation*}
and
\begin{align*}
&\cn[\Gamma_{T,\mathbf{\Pi}_1}(v_1)-\Gamma_{T,\mathbf{\Pi}_2}(v_2);X^s(T)]\\
&\hspace{1cm}\lesssim T^{\frac12} \cn[v_1-v_2;X^s(T)]\{\cn[v_1;X^s(T)]+\cn[v_2;X^s(T)]\}\\
 &\hspace{2cm}+ T^{\frac12}  \|\mathbf{\Pi}_1-\mathbf{\Pi}_2\|\, \cn[v_1;X^s(T)]+ T^{\frac12}  \|\mathbf{\Pi}_2\|\, \cn[v_1-v_2;X^s(T)]+T^{\frac12} \|\mathbf{\Pi}_1-\mathbf{\Pi}_2\| \, , 
\end{align*}
where the proportional constants depending only on $\al$, $s$, and where the norm $\|.\|$ understood as
$$
\|\mathbf{\Pi}\|=\|\mathbf{\Pi}\|_{\mathcal{E}_{\al,p}}\triangleq\cn[\mathbf{\Pi}^{\mathbf{1}} ;L^{\infty}([0,T];\cw^{-\al,p}(\change{\R^d}))]+\cn[\mathbf{\Pi}^{\mathbf{2}} ;L^{\infty}([0,T];\cw^{-2\al,p}(\change{\R^d}))] \ .
$$
\end{proposition}

\

\begin{remark}
Let us briefly compare this result with the situation treated in \cite[Proposition 3.5]{gubinelli-koch-oh}. At the level of the process $\Psi$ (and so at the level of $\mathbf{\Pi}$ in the above formulation), the situation in \cite[Proposition 3.5]{gubinelli-koch-oh} corresponds to taking $\al=\varepsilon$, for $\varepsilon >0$ as small as one wishes. The latter possibility allows the authors of \cite{gubinelli-koch-oh} to consider a general non-linearity of order $k$ in the model (instead of the quadratic non-linearity in (\ref{1-d-quadratic-wave})): morally, the condition $2\al+s<1$ in \eqref{encadr-s} turns into $k\varepsilon+s<1$, which, by taking $\varepsilon$ small enough, can indeed be satisfied. Our aim here, with the result of Proposition \ref{prop:regu-psi-order-two} in mind, is to handle situations where $\al$ may be close to $\frac14$, which accounts for our restriction to a non-linearity of low order.
\end{remark}

\

Just as in the previous section, we easily deduce from Proposition \ref{prop:deterministic-result-d}:
\begin{corollary}\label{cor:deterministic-result-d}
Under the assumptions of Proposition \ref{prop:deterministic-result-d}, and for all (fixed) $(\phi_0,\phi_1) \in \ch^s(\R^d) \times \ch^{s-1}(\R^d)$, $\mathbf{\Pi}=(\mathbf{\Pi}^{\mathbf{1}},\mathbf{\Pi}^{\mathbf{2}})\in \mathcal{E}_{\al,p}$, there exists a time $T_0>0$ such that Equation (\ref{eq-deter-base-bis}) admits a unique solution in $X^s(T_0)$.
\end{corollary}

\smallskip

\begin{proof}[Proof of Proposition \ref{prop:deterministic-result-d}]
Using the estimates in Proposition \ref{prop:regu-wave}, we obtain first
\begin{align}
&\cn[\Gamma_{T,\mathbf{\Pi}_1}(v);X^s(T)] \lesssim \| \phi_0\|_{\ch^s}+\| \phi_1\|_{\ch^{s-1}}+\cn[\rho^2 v^2;L^{2}([0,T];\ch^{s-1}(\R^d))]\nonumber\\
&\hspace{3cm}+\cn[\mathbf{\Pi}^{\mathbf{1}}_1 \cdot \rho v;L^{2}([0,T];\ch^{s-1}(\R^d))]+\cn[\mathbf{\Pi}^{\mathbf{2}}_1 ;L^{2}([0,T];\ch^{s-1}(\R^d))]\, .
\end{align}
Let us now bound the last three quantities separately.

\

\noindent
\textbf{Bound on $\cn[\rho^2 v^2;L^{2}([0,T];\ch^{s-1}(\R^d))]$.} Thanks to condition \eqref{encadr-s}, and using the general embedding result \eqref{sobol-emb}, one can easily check that for $r:=\frac{d}{d-2s}\geq 1$, one has both
\begin{equation}
L^r(D)\subset \ch^{s-1}(D) \quad \text{and} \quad \ch^s(D)\subset L^{2r}(D)\, ,
\end{equation}
which allows us to write
\begin{align*}
\cn[\rho^2 v^2;L^{2}([0,T];\ch^{s-1}(\R^d))] &\lesssim T^{\frac12}\cn[\rho^2 v^2;L^{\infty}([0,T];\ch^{s-1}(D))]\\
& \lesssim T^{\frac12}\cn[\rho^2 v^2 ;L^{\infty}([0,T];L^r(D))]\\
&\lesssim T^{\frac12}\cn[v ;L^{\infty}([0,T];L^{2r}(D))]^2 \ \lesssim T^{\frac12}\cn[v;X^s(T)]^2 \, .
\end{align*}

\

\noindent
\textbf{Bound on $\cn[\mathbf{\Pi}^{\mathbf{1}}_1 \cdot \rho v;L^{2}([0,T];\ch^{s-1}(\R^d))]$.} Let us introduce the additional parameter $1<\rti< 2$ such that
$$\frac{1}{\rti}=\frac12+\frac{1-(\al+s)}{d} \ .$$
Using \eqref{sobol-emb}, one can check that $\cw^{-\al,\rti}(D)\subset \ch^{s-1}(D)$, and so
$$\cn[\mathbf{\Pi}^{\mathbf{1}}_1 \cdot \rho v;L^{2}([0,T];\ch^{s-1}(\R^d))]\lesssim T^{\frac12}\cn[\mathbf{\Pi}^{\mathbf{1}}_1 \cdot \rho v;L^{\infty}([0,T];\cw^{-\al,\rti}(D))]\, .$$
We are now in a position to apply Proposition \ref{product}, which yields
\begin{align*}
\cn[\mathbf{\Pi}^{\mathbf{1}}_1 \cdot \rho v;L^{\infty}([0,T];\cw^{-\al,\rti}(\R^d))]& \lesssim \cn[\mathbf{\Pi}^{\mathbf{1}}_1 ;L^{\infty}([0,T];\cw^{-\al,p}(\R^d))]\, \cn[v;X^s(T)]\\
&\lesssim \|\mathbf{\Pi}_1\|\, \cn[v;X^s(T)]\, .
\end{align*}

\

\noindent
\textbf{Bound on $\cn[\mathbf{\Pi}^{\mathbf{2}}_1 ;L^{2}([0,T];\ch^{s-1}(\R^d))]$.} We know by \eqref{encadr-s} that $s-1<-2\al$, and so, since $\mathbf{\Pi}^{\mathbf{2}}_1$ is compactly supported, we get immediately
$$ \cn[\mathbf{\Pi}^{\mathbf{2}}_1 ;L^{2}([0,T];\ch^{s-1}(\R^d))] \lesssim T^{\frac12}\cn[\mathbf{\Pi}^{\mathbf{2}}_1 ;L^{\infty}([0,T];\cw^{-2\al,p}(\R^d))]\lesssim T^{\frac12}\|\mathbf{\Pi}_1\| \, .$$

\end{proof}

\section{Proof of the main results}\label{sec:proofs-main-theos}

It remains us to combine the (stochastic) results of Section \ref{sec:stochastic} with the (deterministic) results of Section \ref{sec:auxiliary-equation} in order to derive the proof of our main theorems.

\subsection{Proof of Theorem \ref{main-theo}}\label{sec-proof-main-theo}

Fix $d\in \{2,3\}$ and $(\phi_0,\phi_1) \in \ch^1(\R^d) \times L^2(\R^d)$. Now consider the two situations of the statement: 

\

\noindent
$(i)$ Since $\sum_{i=0}^d H_i>d-\frac12$, we can pick $0<\la < \sum_{i=0}^d H_i-d+\frac12$ and then $p$ large enough so that the continuous embedding $\cw^{\la,p}(D) \subset L^\infty(D)$ holds true. By Proposition \ref{prop:regu-psi}, this puts us in a position to apply Corollary \ref{cor:deterministic-result-d-i} (almost surely) with $\mathbf{\Pi}^{\mathbf{1}}\triangleq 2\rho\Psi$, $\mathbf{\Pi}^{\mathbf{2}}\triangleq \rho^2\Psi^2$. The result immediately follows.

\

\noindent
$(ii)$ Pick $\al >0$ such that $d-\frac12-\sum_{i=0}^d H_i <\al < \frac14$, so that, using Proposition \ref{prop:regu-psi} and Proposition \ref{prop:regu-psi-order-two}, one has, for every $p\geq 2$, $\rho\Psi\in L^\infty([0,T];\mathcal{W}^{-\al,p}(\R^d))$ and $\rho^2\widehat{\mathbf{\Psi}}^{\mathbf{2}}\in  L^\infty([0,T];\mathcal{W}^{-2\al,p}(\R^d))$ a.s. The result now follows from the application of Corollary \ref{cor:deterministic-result-d} with $\mathbf{\Pi}^{\mathbf{1}}\triangleq 2\rho\Psi$, $\mathbf{\Pi}^{\mathbf{2}}\triangleq \rho^2\widehat{\mathbf{\Psi}}^{\mathbf{2}}$ and $s= \frac12$.

\

\subsection{Proof of Theorem \ref{main-theo-lim}}\label{sec-proof-main-theo-lim}

Fix $d\in \{2,3\}$ and $(\phi_0,\phi_1) \in \ch^1(\R^d) \times L^2(\R^d)$. Now consider the two situations of the statement: 

\

\noindent
$(i)$ Let $(u_n)$ be the sequence of classical solutions of \eqref{1-d-quadratic-wave-not-renormalized} and set $v_n\triangleq u_n-\Psi_n$, so that for each fixed $n\geq 1$, $v_n$ clearly satisfies Equation \eqref{eq-deter-base-bis} with $\mathbf{\Pi}^{\mathbf{1}}=\mathbf{\Pi}^{\mathbf{1}}_n\triangleq 2\rho\Psi_n$ and $\mathbf{\Pi}^{\mathbf{2}}=\mathbf{\Pi}^{\mathbf{2}}_n\triangleq \rho^2\Psi_n^2$. We can thus apply \eqref{boun-gamma-1} and assert that for every $T>0$,
\begin{equation}\label{boun-v-n-1}
\cn[v_n;X(T)] \lesssim\| \phi_0\|_{\ch^1}+\| \phi_1\|_{L^2}+T^{\frac12} \cn[v_n;X(T)]^2+ T^{\frac12}  \|\mathbf{\Pi}_n\|\, \cn[v_n;X(T)]+T^{\frac12}  \|\mathbf{\Pi}_n\| \ ,
\end{equation}
where the proportional constant only depends on $\rho$. Besides, using a standard Sobolev embedding, we know that for all $\la >0$ and $p$ large enough (depending on $\la$), 
\begin{equation}\label{boun-psi-n-psi}
\|\mathbf{\Pi}_n-\mathbf{\Pi}\|=\|\mathbf{\Pi}_n-\mathbf{\Pi}\|_{\mathcal{E}} \lesssim   \|\rho\Psi_n-\rho\Psi\|_{L^\infty([0,T];\cw^{\la,p}(D))} +\|\rho^2\Psi_n^2-\rho^2\Psi^2\|_{L^\infty([0,T];\cw^{\la,p}(D))} \ ,
\end{equation}
and so, using Proposition \ref{prop:regu-psi}, we get that for a subsequence of $(\rho\Psi_n)$ (that we still denote by $(\rho\Psi_n)$), $\|\mathbf{\Pi}_n-\mathbf{\Pi}\| \to 0$ almost surely. In particular, $\sup_n \|\mathbf{\Pi}_n\| < \infty$ a.s. Going back to \eqref{boun-v-n-1} and setting $f_n(T) \triangleq \cn[v_n;X(T)]$, we deduce that for all $0<T_0\leq 1$ and $0<T\leq T_0$,
$$
f_n(T) \leq C_1 \big\{ A+T_0^{1/2} f_n(T)^2\big\}  \ ,
$$
for some (random) constant $C_1\geq 1$, and where we have set $A\triangleq 1+\| \phi_0\|_{\ch^1}+\| \phi_1\|_{L^2}$. At this point, let us fix the (random) time $T_1>0$ satisfying $1-4C_1^2 A T_1^{1/2}=\frac12$, in such a way that for every $0<T_0\leq T_1$, the equation $C_1T_0^{1/2} x^2-x+C_1 A=0$ admits two solutions $x_{1,T_0},x_{2,T_0}$ satisfying $0<x_{1,T_0}<x_{2,T_0}$. As a result, for all $0<T_0 \leq \inf(1,T_1)$ and $0<T\leq T_0$, one has either $f_n(T) \leq x_{1,T_0}$ or $f_n(T) \geq x_{2,T_0}$. In fact, due to the continuity of $T \mapsto f_n(T)$ (a straighforward consequence of the regularity of $v_n$), one has either $\sup_{T\in [0,T_0]} f_n(T) \leq x_{1,T_0}$ or $\inf_{T\in [0,T_0]} f_n(T) \geq x_{2,T_0}$. Moreover, if we define $T_2>0$ as the largest time such that $ (\sup_{0\leq T \leq T_2} 2C_1T^{1/2} \|\phi_0\|_{\ch_1}) \leq 1$, it can be explicitly checked that for every $0 <T_0\leq \inf(1,T_1,T_2)$, $f_n(0)=\|\phi_0\|_{\ch^1} \leq x_{1,T_0}$, and we are therefore in a position to assert that for such a time $T_0$ (that we fix from now on), 
$$\sup_{n\geq 1} \cn[v_n ;X(T_0)] =\sup_{n\geq 1} \sup_{T\in [0,T_0]} f_n(T) \leq  x_{1,T_0} \ .$$
By injecting this uniform bound into (\ref{boun-gamma-2}), we easily derive that, for some time $0<\bar{T}_0 \leq T_0$ (uniform in $n$), 
$$\cn[v_n-v;X^s(\bar{T}_0)] \lesssim \| \mathbf{\Psi}-\mathbf{\Psi}_n\| \ ,$$
where the proportional constant is also uniform in $n$. Finally,
$$
\cn[u_n-u;L^\infty([0,\bar{T}_0];L^2(D))] \lesssim \cn[\Psi_n-\Psi;L^\infty([0,\bar{T}_0];L^2(D))] +\cn[v_n-v;X(\bar{T}_0)] 
$$
and the convergence immediately follows.

\

\noindent
$(ii)$ We can of course use the very same arguments as above, by noting that if $(u_n)$ satisfies \eqref{1-d-quadratic-wave-renormalized-lim} and $v_n \triangleq v_n-\Psi_n$, then $v_n$ satisfies \eqref{eq-deter-base-bis} with $\mathbf{\Pi}^{\mathbf{1}}=\mathbf{\Pi}^{\mathbf{1}}_n\triangleq 2\rho\psi_n$ and $\mathbf{\Pi}^{\mathbf{2}}=\mathbf{\Pi}^{\mathbf{2}}_n\triangleq \rho^2\widehat{\mathbf{\Psi}}^{\mathbf{2}}_n$. The bound \eqref{boun-psi-n-psi} must then be replaced with 
\begin{multline*}
\| \mathbf{\Pi}_n-\mathbf{\Pi}\|=\| \mathbf{\Pi}_n-\mathbf{\Pi}\|_{\mathcal{E}_{\al,p}}=\\
\cn[2\rho\Psi_n-2\rho\Psi ;L^{\infty}([0,T];\cw^{-\al,p}(D))]+\cn[\rho^2\widehat{\mathbf{\Psi}}^{\mathbf{2}}_n-\rho^2\widehat{\mathbf{\Psi}}^{\mathbf{2}} ;L^{\infty}([0,T];\cw^{-2\al,p}(D))] \ ,
\end{multline*}
which allows us to apply Proposition \ref{prop:regu-psi} and Proposition \ref{prop:regu-psi-order-two} in the procedure. Observe finally that
\begin{align*}
&\cn[u_n-u;L^\infty([0,T];\ch^{-\al}(D))] \lesssim \cn[\Psi_n-\Psi;L^\infty([0,T];\ch^{-\al}(D))] +\cn[v_n-v;X^\frac12(T)] \ ,
\end{align*}
which yields the expected convergence.

\

\


\begin{thebibliography}{99}


\bibitem{balan}
R.M. Balan: The Stochastic Wave Equation with Multiplicative Fractional Noise: A Malliavin Calculus Approach. {\it Potential Analysis}, {\bf 36} (2012), no. 1, 1--34.

\smallskip

\bibitem{balan-jolis-quer}
R.M. Balan, M. Jolis and L. Quer-Sardanyons: SPDEs with affine multiplicative fractional noise in space with index $1/4<H<1/2$. {\it Electron. J. Probab.} {\bf 20} (2015), no. 54, 36 pp.

\smallskip

\bibitem{balan-tudor}
R.M. Balan and C.A. Tudor: The stochastic wave equation with fractional noise: a random field approach. {\it Stoch. Proc. Their Appl.} {\bf 120} (2010), 2468-2494.


\smallskip

\bibitem{bourgain}
J. Bourgain: Invariant measures for the 2D-defocusing nonlinear Schr{\"o}dinger equation. {\it Comm. Math. Phys.} {\bf 176} (1996), no. 2, 421-445.

\smallskip

\bibitem{burq-tzvetkov} 
N. Burq and N. Tzvetkov: Random data Cauchy theory for supercritical wave equations I: local existence theory. {\it Invent. Math.} {\bf 173} (2008), no. 3, 449-475.

\smallskip

\bibitem{caithamer}
P. Caithamer: The stochastic wave equation driven by fractional Brownian noise and temporally correlated smooth noise. {\it Stoch. Dyn.} {\bf 5} (2005), no. 1, 45-64.

\smallskip




\smallskip

\bibitem{daprato-zabczyk}
G. Da Prato and J. Zabczyk: Stochastic equations in infinite dimensions. Encyclopedia of Mathematics and its Applications, 152. Cambridge University Press, Cambridge, 2nd edition, 2014.

\smallskip

\bibitem{deya}
A. Deya: On a modelled rough heat equation. {\it Probab. Theory Related Fields} {\bf 166} (2016), no. 1, 1-65.

\smallskip

\bibitem{deya2}
A. Deya: Construction and Skorohod representation of a fractional $K$-rough path. {\it Electron. J. Probab.}, {\bf 22} (2017), no. 52.

\smallskip

\bibitem{deya-wave-2}
A. Deya: On a non-linear 2D fractional wave equation. {\it Ann. Inst. H. Poincar{\'e} Probab. Statist.}, 56(1), 477-501, 2020.

\smallskip

\bibitem{RHE}
A. Deya, M. Gubinelli and S. Tindel: Non-linear rough heat equations.  {\it Probab. Theory Related Fields}  {\bf 153}  (2012),  no. 1-2, 97--147.


\smallskip

\bibitem{erraoui-ouknine-nualart}
M. Erraoui, Y. Ouknine and D. Nualart: Hyperbolic Stochastic Partial Differential Equations with Additive Fractional Brownian Sheet. {\it Stoch. Dyn.} {\bf 3} (2003), no. 121. 

\smallskip
 
\bibitem{ginibre-velo}
J. Ginibre and G. Velo: Generalized Strichartz inequalities for the wave equation, {\it J. Funct. Anal.} {\bf 133} (1995), 50-68.

\smallskip

\bibitem{gubinelli-koch-oh}
M. Gubinelli, H. Koch and T. Oh: Renormalization of the two-dimensional stochastic non linear wave equations. {\it Trans. Amer. Math. Soc.} {\bf 370} (2018), 7335-7359.

\smallskip

\bibitem{GLT}
M. Gubinelli, A. Lejay and S. Tindel: Young integrals and SPDEs. {\it Pot. Anal.} {\bf 25} (2006), no. 4, 307--326.  

\smallskip

\bibitem{REE}
M. Gubinelli and S. Tindel: Rough evolution equations.  {\it Ann. Probab.}  {\bf 38} (2010), no. 1, 1--75.

\smallskip

\bibitem{hai-14}
M. Hairer: A theory of regularity structures. {\it Invent. Math.} {\bf 198} (2014), no. 2, 269--504.

\smallskip

\bibitem{hu}
Y. Hu: Heat equation with fractional white noise potentials. {\it Appl. Math. Optim.} {\bf 43} (2001), 221-243.

\smallskip

\bibitem{hu-lu-nualart}
Y. Hu, F. Lu and D. Nualart: Feynman-Kac formula for the heat equation driven by fractional noise with Hurst parameter $H<1/2$. {\it Ann. Probab.} {\bf 40} (2012), 1041-1068.

\smallskip

\bibitem{hu-nualart-song}
Y. Hu, D. Nualart and J. Song: Feynman-Kac formula for heat equation driven by fractional white noise. {\it Ann. Probab.} {\bf 39} (2011), 291-326.

\smallskip



\bibitem{oh-thomann}
T. Oh and L. Thomann: Invariant Gibbs measure for the 2-d defocusing nonlinear wave equations. {\it Ann. Fac. Sci. Toulouse Math.} {\bf 6} 29 (2020), no. 1, 1-26.      

\smallskip

\bibitem{quer-tindel}
L. Quer-Sardanyons and S. Tindel: The 1-d stochastic wave equation driven by a fractional Brownian motion. {\it Stoch. Proc. Their Appl.}, {\bf 117} (2007), 1448-1472. 

\smallskip

\bibitem{runst-sickel}
T. Runst and W. Sickel: Sobolev Spaces of Fractional Order, Nemytskij Operators, and Nonlinear Partial Differential Equations. de Gruyter Series in Nonlinear Analysis and Applications 3, Berlin (1996).

\smallskip

\bibitem{samo-taqqu}
G. Samorodnitsky and M. S. Taqqu: Stable non-Gaussian random processes. Chapman and Hall, 1994.


\smallskip

\bibitem{thomann}
L. Thomann: Random data Cauchy problem for supercritical Schr{\"o}dinger equations. {\it Ann. I. H. Poincar{\'e} - AN}, {\bf 26} (2009), no. 6, 2385--2402. 

\smallskip

\bibitem{tindel-tudor-viens}
S. Tindel, C. Tudor and F. Viens: Stochastic evolution equations with fractional Brownian motion. {\it Probab. Theory Related Fields}, {\bf 127} (2003) 186-204.

\smallskip

\bibitem{walsh}
J.B. Walsh: An introduction to stochastic partial differential equations. In: {\it {\'E}cole d'{\'e}t{\'e} de probabilit{\'e}s de Saint-Flour}, XIV-1984. Lecture Notes in Mathematics, vol. 1180, pp. 265-439. Springer, Berlin (1986).


\end{thebibliography}
\end{document}